\documentclass[11pt]{amsart}
\usepackage{amssymb}
\usepackage{amsthm}
\usepackage{mathrsfs}
\usepackage{stmaryrd}
\usepackage{enumitem}
\usepackage[french,english]{babel}
\usepackage[utf8]{inputenc}
\usepackage[T1]{fontenc}
\usepackage{lmodern}
\usepackage[all]{xy}
\usepackage[colorlinks=true]{hyperref}
\hypersetup{colorlinks, citecolor=blue, filecolor=black, linkcolor=black, urlcolor=black}
\usepackage{upref}
\usepackage{array}
\usepackage{color}
\usepackage{tikz-cd}
\usepackage{verbatim}
\usepackage{framed}
\usepackage[capitalize]{cleveref}

\usepackage{showlabels}

\makeatletter\def\SL@eqnlefttext #1{\hbox to 0pt{\kern 60pt 
		\llap{\SL@margintext{#1}\quad}\hss}}
\makeatother

\DeclareMathOperator{\Be}{Be}
\DeclareMathOperator{\Kl}{Kl}

\newtheorem{theorem}{Theorem}[section]
\newtheorem{prop}[theorem]{Proposition}
\newtheorem{lemma}[theorem]{Lemma}

\newtheorem{coro}[theorem]{Corollary}

\theoremstyle{definition}
\newtheorem{rem}[theorem]{Remark}
\newtheorem{definition}[theorem]{Definition}

\newtheorem{con}[theorem]{Construction}
\newtheorem{nota}[theorem]{Notation}

\numberwithin{equation}{section}
\numberwithin{equation}{theorem}

\newcommand{\quash}[1]{}

\title{comparison between formal slopes and $p$-adic slopes}
\author{Yezheng Gao }
\date{}

\begin{document}
	\selectlanguage{english}
	
	\maketitle

	\textbf{Abstract:} In this paper, we establish several inequalities comparing formal slopes with $p$-adic slopes of solvable differential modules over the punctured open unit disc. Our approach is based on a delicate analysis of Newton polygons and the log-convexity of generic radius functions.

	
	\section{Introduction}
	
	The study of differential modules over varieties is of particular importance. Local variants of these objects in the curve case are differential modules over formal punctured discs or rigid annuli. In this paper, we prove a comparison between certain ramification invariants associated with these local objects.
	
	On the one hand, let $k$ be a field of characteristic $0$ and $K=k((x))$. In the formal theory of differential modules over $K$ relative to $k$, a seminal result is the Turrittin-Levelt decomposition theorem (see \cite{Le75,Tu55}). This decomposition enables us to associate some ramification invariants called \textit{formal slopes} to a differential module over $K$.
	
	On the other hand, let $p$ be a prime and $k$ be a complete discrete valuation field of mixed characteristic $(0,p)$. The Robba ring $\mathcal{R}$ is the ring of analytic functions over $k$ on an open annulus of outer radius $1$ and unspecific inner radius. Christol, Dwork and Mebkhout systematically studied differential modules over $\mathcal{R}$ relative to $k$ in \cite{C-D94,C-M93,C-M97,C-M00,C-M01}. Our research focuses on solvable differential modules as defined in \cite{C-M00}, to which we can associate some ramification invariants called $p$\textit{-adic slopes}.
	
	We relate the theory of formal slopes to that of $p$-adic slopes by considering a subring $\mathcal{A}_x$ of $\mathcal{R}$, consisting of analytic functions on the punctured open unit disc over $k$. Given a solvable differential module $M$ over $\mathcal{A}_x$, since $\mathcal{A}_x$ can be naturally embedded into $k((x))$, both formal slopes and $p$-adic slopes of $M$ are well defined.
	
	It is natural to compare these two slopes. According to a result of Baldassarri in \cite{Ba82}, the maximum $p$-adic slope is less than or equal to the maximum formal slope. In this paper, we show the following inequalities. Note that the residue field of $k$ is not necessarily perfect in our setting.
	
	\begin{theorem}
		Let $M$ be a solvable differential module of rank $n$ over $\mathcal{A}_x$. Let
		\[\alpha_1\geq\cdots\geq\alpha_n\quad\mathrm{(resp.}\quad\beta_1\geq\cdots\geq\beta_n\mathrm{)}\]
		be the $p$-adic slopes (resp. formal slopes) of $M$ listed in the decreasing order. Then for each $1\leq i\leq n$, the following inequality holds.
		\begin{equation} 
			\sum_{j=1}^{i}\alpha_j\leq\sum_{j=1}^{i}\beta_j.
		\end{equation}
	\end{theorem}
	
	\begin{rem}
		The partial-sum inequality (1.1.1) can also be deduced from the theory of convergence Newton polygons developed by Poineau and Pulita \cite{Pu15,PP15} (see Proposition 3.4.1 of \cite{PP24}). The purpose of this paper is to give a direct proof for solvable differential modules over the punctured open disc that avoids Berkovich geometry and makes the relation with formal Newton polygons explicit. Our main technical input is a small-radius analysis of Newton polygons attached to a cyclic presentation over certain annulus, from which we derive explicit formulas for computing subsidiary generic radii (see Lemma 5.12 and 5.14).
	\end{rem}
	
	\begin{rem}
		The inequality (1.1.1) might be strict in general. For example, given an element $a(x)\in\mathcal{A}_x$, let $\mathrm{exp}\Big(a(x)\Big)$ denote the rank one differential module associated to the differential equation
		\[\frac{d}{dx}-a'(x)=0,\]
		where $a'(x)$ denotes the derivative of $a(x)$.
		
		Suppose that $k$ contains an element $\pi$ such that $\pi^{p-1}=-p$, then the exponential module \[M=\mathrm{exp}(\frac{\pi}{x^{p^n}}),\quad n\in\mathbb{N}\]
		is solvable. The formal slope (resp. $p$-adic slope) of $M$ is $p^n$ (resp. $1$).
	\end{rem}
	
	\begin{rem}
		Given a solvable differential module $M$ over $\mathcal{R}$, there exists some differential module $M'$ over $\mathcal{A}_x$ such that $M\simeq M'\otimes_{\mathcal{A}_x}\mathcal{R}$ (\cite{C-M01}, Theorem 4.2-7). We call $M'$ a model of $M$ over $\mathcal{A}_x$. For different models $M'$, the formal slopes of $M'$ may vary, but the $p$-adic slopes are all equal to those of $M$.
		
		Christol and Mebkhout asked whether there exists a model $M'$ such that
		\begin{equation}
			\sum_{j=1}^{n}\alpha_j=\sum_{j=1}^{n}\beta_j,
		\end{equation}
		where $\alpha_1,\cdots,\alpha_n$ (resp. $\beta_1,\cdots,\beta_n$) are the $p$-adic slopes (resp. formal slopes) of $M'$ and $n$ is the rank of $M$ (\cite{C-M01}, Problem 5.0-14).
		
		The rank one case is already proved by classifying solvable differential modules in \cite{Pu04}.
		
		For instance, let $M$ be the exponential module $\mathrm{exp}(\frac{\pi}{x^p})$ over $\mathcal{R}$. Then $M'=\mathrm{exp}(\frac{\pi}{x})$ is the desired model over $\mathcal{A}_x$, whose formal slopes and $p$-adic slopes are both equal to $1$. The overconvergent Dwork exponential series (see \cite{Rob00}, Chapter 7.2.4) induces an isomorphism $M\simeq M'\otimes_{\mathcal{A}_x}\mathcal{R}$.
		
		However, even for a model satisfying (1.4.1), the inequality (1.1.1) might still be strict when $i<n$. In Section 6.4, we present an example of this phenomenon.
	\end{rem}
	
	The proof of Theorem 1.1 is inspired by the work of Baldassarri in \cite{Ba82}. The information of formal slopes and $p$-adic slopes is encoded in certain functions $F_i(M,r)$, which are defined by generic radii in Section 4. These functions are piecewise affine, continuous and convex on the interval $(0,\infty)$ (see Figure 1). On the one hand, the slope of $F_i(M,r)$ is $i+\sum_{j=1}^{i}\alpha_j$ as $r\rightarrow0$. On the other hand, we show that the slope of $F_i(M,r)$ is $i+\sum_{j=1}^{i}\beta_j$ as $r\rightarrow\infty$ in Section 5 through a delicate analysis of Newton polygons. Finally, we conclude the theorem by the convexity of $F_i(M,r)$.
	
	In Section 2, we review some general facts about differential modules. In Section 3, we revise an algorithm for computing formal slopes (\cite{Ka87}, (2.2.10)). We then rewrite the algorithm using the language of formal Newton polygons. In Section 4, we review the notions of generic radii, Newton polygons, Robba rings, solvable modules and the $p$-adic slopes. This section is mostly based on Chapter 9\textasciitilde12 of Kedlaya's work \cite{Ked22}. In Section 5, we give a detailed proof of Theorem 1.1 as sketched above. In Section 6, we present some examples from Bessel equations and study the associated monodromy representations. In particular, the phenomenon at the end of Remark 1.4 is explained using an example.
	
	\begin{center}
		\begin{figure}[htbp]
			\begin{tikzpicture}
				\draw[->] (0,0) -- (6,0) node[right] {$r$};
				\draw[->] (0,0) -- (0,3) node[above] {$F_i(M,r)$};
				\node[above left] at (0,0) {$0$};
				\filldraw[black] (0,0) circle (1pt);
				\filldraw[black] (2.5,0.5) circle (1pt);
				\filldraw[black] (4,1.5) circle (1pt);
				\draw[black] (0,0)--(2.5,0.5)--(4,1.5)--(5,2.8);
			\end{tikzpicture}
			\caption{}
		\end{figure}
	\end{center}
	
	\subsection*{Acknowledgements}
	This article is a part of my thesis prepared at Academy of Mathematics and Systems Science. I would like to express my great gratitude to my thesis advisor Daxin Xu for leading
	me to this question and for his helpful comments on earlier versions of this work. I would also like to thank Andrea Pulita for his discussions and encouragement.
	
    The author is supported by National Key R\&D Program of China (2025YFA1018000) and National Natural Science Foundation of China (12222118).
	
	\section{General facts about differential modules}
	
	\subsection{Basic definitions}
	
	\begin{definition}
		(1) Let $R$ be a commutative ring with $1$. A \textit{derivation} over $R$ is an additive map $d:R\rightarrow R$ satisfying the Leibniz rule
		\[d(ab)=ad(b)+bd(a),\quad\forall~a,b\in R.\]
		A \textit{differential ring} is a ring $R$ equipped with a derivation $d$, denoted by $(R,d)$ or simply $R$. If $R$ is a field in addition, we call it a \textit{differential field}.
		
		(2) A \textit{differential module} over a differential ring $(R,d)$ is an $R$-module $M$ equipped with an additive map $D:M\rightarrow M$ that satisfies the following property. 
		\[D(am)=aD(m)+d(a)m,\quad\forall~a\in R,m\in M.\]
		We denote the differential module by $(M,D)$ or simply $M$, and $D$ is called a \textit{differential operator} on $M$ relative to $d$.
		
		A \textit{homomorphism} between differential modules over $R$ is an $R$-linear map that is compatible with the differential operators.
	
		(3) Let $R$ be a differential ring and $(M_1,D_1),(M_2,D_2)$ be differential modules over $R$. Define a differential operator $D$ on $M_1\otimes_FM_2$ by
		\[D(a\otimes b)=D_1(a)\otimes b+a\otimes D_2(b),\quad\forall~a\in M_1,b\in M_2.\]
		$(M_1\otimes_FM_2,D)$ is called the \textit{tensor product} of $M_1$ and $M_2$.
	\end{definition}
	
	\subsection{Cyclic basis and twisted polynomials}
	
	\begin{definition}
		Let $(M,D)$ be a free differential module of rank $n$ over $R$. A \textit{cyclic vector} of $M$ is an element $m\in M$ such that $m,D(m),\cdots,D^{n-1}(m)$ form an $R$-basis of $M$. A \textit{cyclic basis} of $M$ is a basis of this form.
	\end{definition}
	
	\begin{theorem}
		$\mathrm{(}$\cite{D-G-S94}$\mathrm{,~Theorem~III.4.2)}$ Let $(F,d)$ be a differential field of characteristic $0$. If $d\neq0$, then every finite-dimensional differential module $M$ over $F$ has a cyclic basis.
	\end{theorem}
	
	\begin{definition}
		(Ore) Let $(R,d)$ be a differential ring. The \textit{ring of twisted polynomials} $R\{T\}$ over $R$ is defined as the free $R$-module of formal summations
		\[\Big\{\sum_{i=0}^{n}a_iT^i:a_i\in R,n\in \mathbb{N}\Big\},\]
		equipped with the noncommutative multiplication induced by
		\[T\cdot a=aT+d(a),\quad\forall~a\in R.\]
	\end{definition}
	
	\begin{rem}
		Let $(M,D)$ be a free differential module of rank $n$ over $R$. If $M$ has a cyclic vector $m$, then there exist $a_0,\cdots,a_{n-1}\in R$ such that
		\[D^n(m)+\sum_{i=0}^{n-1}a_iD^i(m)=0,\quad D^0=\mathrm{id}_M.\]
		
		Let $\ell=T^n+\sum_{i=0}^{n-1}a_iT^i\in R\{T\}$, we have an isomorphism of differential modules $M\simeq R\{T\}/R\{T\}\ell$, where $D^i(m)$ is sent to $T^i$, $i=0,1\cdots,n-1$.
	\end{rem}
	
	\section{Formal theory}
	
	In this chapter, let $k$ be a field of characteristic $0$ and $K=k((x))$ be the field of Laurent formal power series over $k$. Equip $K$ with the derivation $\frac{d}{dx}$ relative to $k$.
	
	\subsection{An algorithm for computing formal slopes}

	\begin{definition}
		The $x$-\textit{adic valuation}, denoted by $\mathrm{ord}_x$, is a discrete valuation on $K$ defined as follows.
		\[\mathrm{ord}_x:K\rightarrow\mathbb{Z}\cup\{\infty\},\quad\mathrm{ord}_x\Big(\sum_{i=n}^{\infty}a_ix^i\Big)=n,\quad a_i\in k,a_n\neq0.\]
	\end{definition}
	
	Let $M$ be a finite-dimensional differential module over $(K,\frac{d}{dx})$. The formal slopes of $M$ can be defined as a finite sequence of non-negative rational numbers (see \cite{Ka87}, Section 2.2). Rather than providing the exact definition, we present an algorithm for computing formal slopes, as this is sufficient for our proof.
	
	Since $M$ has a cyclic basis according to Theorem 2.3, we can assume that $M\simeq K\{T\}/K\{T\}\ell$, where $\ell$ is a twisted polynomial as described in Remark 2.5.
	
	\begin{prop}
		$\mathrm{(}$\cite{Ka87}$\mathrm{,~(2.2.10))}$ Let $\ell=\sum_{i=0}^{n}a_iT^i$ be a non-constant, monic twisted polynomial in $K\{T\}$, and let $f_\ell(t)=\sum_{i=0}^{n}a_it^i$ be the corresponding polynomial in $K[t]$. Fix an algebraic closure $\overline{K}$ of $K$ and factor $f_\ell(t)$ in $\overline{K}[t]$ as
		\[f_\ell(t)=\prod_{i=1}^{n}(t-\mu_i).\]
		Then the multi-set of formal slopes of $K\{T\}/K\{T\}\ell$ is
		\[\{\max\{0,-1-\mathrm{ord}_x(\mu_i)\}:i=1,\cdots,n\},\]
		where $\mathrm{ord}_x$ is the extension of the $x$-adic valuation of $K$ to $\overline{K}$.
	\end{prop}
	
	\subsection{Formal Newton polygons}
	
	\begin{definition}
		Let $\ell=\sum_{i=0}^{n}a_iT^i$ be a monic twisted polynomial. 
		
		(1) The \textit{formal Newton polygon} of $\ell$, denoted by $\mathrm{FNP}(\ell)$, is defined as the boundary of the lower convex hull of the following set in $\mathbb{R}^2$ (see Figure 2).
		\[\{(-i,\mathrm{ord}_x(a_i))\in\mathbb{R}^2:i=0,1,\cdots,n\}.\]
		
		(2) An integer $m$ is called a \textit{break} of $\mathrm{FNP}(\ell)$ if either $m$ belongs to the set $\{-n,0\}$, or $-n<m<0$ and the slope of $\mathrm{FNP}(\ell)$ on the interval $[m-1,m]$ is strictly smaller than that on the interval $[m,m+1]$.
		
		(3) Let $-n=m_0<m_1\cdots<m_k=0$ denote all the breaks of $\mathrm{FNP}(\ell)$, and let \[\lambda_i=\frac{\mathrm{ord}_x(a_{-m_i})-\mathrm{ord}_x(a_{-m_{i-1}})}{m_i-m_{i-1}}\]
		be the slope of $\mathrm{FNP}(\ell)$ on the interval $[m_{i-1},m_i]$ for $i=1,\cdots,k$. The \textit{multi-set of slopes} of $\mathrm{FNP}(\ell)$ is defined as
		\[\mathrm{FS}(\ell)=\{\lambda_i \ \mathrm{with} \ \mathrm{multiplicity} \ m_i-m_{i-1}:i=1,\cdots,n\}.\]
		
		(4) The \textit{multi-set of effective slopes} of $\mathrm{FNP}(\ell)$, denoted by $\mathrm{EFS}(\ell)$, is defined as the subset of $\mathrm{FS}(\ell)$ consisting of slopes $<-1$.
	\end{definition}
	
	\begin{center}
		\begin{figure}[htbp]
			\begin{tikzpicture}
				\draw[->] (-5,0) -- (0.5,0) node[right] {$x$};
				\draw[->] (0,-3) -- (0,0.5) node[above] {$y$};
				\node[below left] at (0,0) {$0$};
				\node[above] at (-4,0) {$-n$};
				\filldraw[black] (-4,0) circle (1pt); 
				\filldraw[black] (-3,-1.8) circle (1pt); 
				\filldraw[black] (-1,-2.8) circle (1pt); 
				\filldraw[black] (0,-1) circle (1pt); 
				\filldraw[black] (-3.5,-0.5) circle (1pt);
				\filldraw[black] (-2.5,0.3) circle (1pt);
				\filldraw[black] (-2,-1.9) circle (1pt);
				\filldraw[black] (-1.5,-1.5) circle (1pt);
				\filldraw[black] (-0.5,-1.9) circle (1pt);
				\draw[black] (-4,0) -- (-3,-1.8) -- (-1,-2.8) -- (0,-1) ;
			\end{tikzpicture}
			\caption{}
		\end{figure}
	\end{center}
	
	\begin{lemma}
		Let $\ell=\sum_{i=0}^{n}a_iT^i$ be a non-constant, monic twisted polynomial and let $m$ be the cardinality of $\mathrm{EFS}(\ell)$. Then the multi-set of formal slopes of $K\{T\}/K\{T\}\ell$ is
		\begin{equation}
			\{-\lambda-1:\lambda \ \mathrm{runs} \ \mathrm{through} \ \mathrm{EFS}(\ell)\}\cup\{0 \ \mathrm{with} \ \mathrm{multiplicity} \ n-m\}.
		\end{equation} 
	\end{lemma}
	
	\begin{proof}
		Let $f_\ell(t)=\sum_{i=0}^{n}a_it^i$ be the corresponding polynomial in $K[t]$. Fix an algebraic closure $\overline{K}$ of $K$ and factor $f_\ell(t)$ in $\overline{K}[t]$ as
		\[f_\ell(t)=\prod_{i=1}^{n}(t-\mu_i).\]
		According to \cite{Ked22} Proposition 2.1.5, the multi-set
		\[\{\mathrm{ord}_x(\mu_1),\cdots,\mathrm{ord}_x(\mu_n)\}\]
		is equal to $\mathrm{FS}(\ell)$.
		
		Combining this fact with Proposition 3.2, we complete the proof.
	\end{proof}
	
	\section{$P$-adic theory}
	
	In this section, let $p$ be a prime and $(k,|\cdot|)$ be a complete discrete valuation field of mixed characteristic $(0,p)$. Normalize the absolute value so that $|p|=p^{-1}$.
	
	\subsection{Generic radius and strong decomposition}
	
	\begin{definition}
		Let $\rho>0$. The $\rho$-\textit{Gauss norm} $|\cdot|_\rho$ on the polynomial ring $k[x]$ is defined by
		\[|\cdot|_\rho:k[x]\rightarrow\mathbb{R}_{\geq0},\quad\Big|\sum_{i=0}^{n}a_ix^i\Big|_\rho=\max_{0\leq i\leq n}\{|a_i|\rho^i\}.\]
		Let $F_\rho$ denote the completion of $k(x)$ with respect to $|\cdot|_\rho$.
	\end{definition}
	
	\begin{definition}
		Let $(M,D)$ be a finite-dimensional differential module over $(F_\rho,\frac{d}{dx})$ with a basis $e_1,\cdots,e_n$.
		
		(1) Equip $M$ with a norm $||\cdot||$ compatible with $|\cdot|_\rho$ as follows.
		\[||\cdot||:M\rightarrow\mathbb{R}_{\geq0},\quad||\sum_{i=1}^{n}a_ie_i||=\max_{1\leq i\leq n}\{|a_i|_\rho\},\quad a_i\in F_\rho.\]
		
		(2) The \textit{spectral radius} (resp. \textit{generic radius}) of $D$ is defined by
		\begin{equation}
			|D|_{\mathrm{sp},M}=\lim_{i\rightarrow\infty}|D^i|_M^\frac{1}{i}\quad\mathrm{(resp.}\quad R(M)=p^{-\frac{1}{p-1}}|D|_{\mathrm{sp},M}^{-1}\mathrm{)},
		\end{equation}
		where $|D^i|_M$ is the operator norm of $D^i$ with respect to $||\cdot||$.
	\end{definition}
	
	\begin{rem}
		The spectral radius $|D|_{\mathrm{sp},M}$ is independent of the choice of a basis of $M$. The existence of the limit in (4.2.1) is guaranteed by Fekete's lemma (\cite{Ked22}, Lemma 6.1.4).
	\end{rem}
	
	\begin{theorem}
		$\mathrm{(}$\cite{Ked22}$\mathrm{,~Theorem~10.6.2,~Strong~Decomposition)}$ Let $(M,D)$ be a finite-dimensional differential module over $(F_\rho,\frac{d}{dx})$. There exists a unique decomposition of differential modules
		\begin{equation}
			M\simeq\bigoplus\limits_{s\in (0,\rho]}M_s,
		\end{equation}
		such that
		
		(1) $M_s=0$ for all but finitely many $s\in(0,\rho]$.
		
		(2) If $M_s\neq0$ for some $s\in(0,\rho]$, then the generic radius of the restriction of $D$ to each submodule or quotient module of $M_s$ is equal to $s$. 
	\end{theorem}
	
	\begin{definition}
		Let $M$ be a finite-dimensional differential module over $(F_\rho,\frac{d}{dx})$ and let (4.4.1) be its strong decomposition. The \textit{multi-set of subsidiary radii} of $M$ is defined as
		\[\{s \ \mathrm{with} \ \mathrm{multiplicity} \ \mathrm{dim}_{F_\rho}M_s:s\in(0,\rho],M_s\neq0\}.\]
	\end{definition}
	
	\subsection{Newton polygons}
	
	Equip $F_\rho$ with a valuation
	\[v_\rho:F_\rho\rightarrow\mathbb{R}\cup\{\infty\},\quad v_\rho=-\mathrm{log}|\cdot|_\rho,\]
	where $\mathrm{log}$ denotes the natural logarithm.
	
	\begin{definition}
		Let $\ell=\sum_{i=0}^{n}a_iT^i$ be a monic twisted polynomial in $F_\rho\{T\}$. 
		
		(1) The \textit{Newton polygon} of $\ell$, denoted by $\mathrm{NP}_\rho(\ell)$, is defined as the boundary of the lower convex hull of the following set in $\mathbb{R}^2$ (see Figure 2).
		\[\{(-i,v_\rho(a_i))\in\mathbb{R}^2:i=0,1,\cdot,n\}.\]
		
		(2) An integer $m$ is called a \textit{break} of $\mathrm{NP}_\rho(\ell)$ if either $m$ belongs to the set $\{-n,0\}$, or $-n<m<0$ and the slope of $\mathrm{NP}_\rho(\ell)$ on the interval $[m-1,m]$ is strictly smaller than that on the interval $[m,m+1]$.
		
		(3) Let $-n=m_0<m_1\cdots<m_k=0$ denote all the breaks of $\mathrm{NP}_\rho(\ell)$, and let \[\lambda_i=\frac{v_\rho(a_{-m_i})-v_\rho(a_{-m_{i-1}})}{m_i-m_{i-1}}\]
		be the slope of $\mathrm{NP}_\rho(\ell)$ on the interval $[m_{i-1},m_i]$ for $i=1,\cdots,k$. The \textit{multi-set of slopes} of $\mathrm{NP}_\rho(\ell)$ is defined as
		\[\mathrm{S}_\rho(\ell)=\{\lambda_i \ \mathrm{with} \ \mathrm{multiplicity} \ m_i-m_{i-1}:i=1,\cdots,n\}.\]
		
		(4) The \textit{multi-set of effective slopes} of $\mathrm{NP}_\rho(\ell)$, denoted by $\mathrm{ES}_\rho(\ell)$, is defined as the subset of $\mathrm{S}_\rho(\ell)$ consisting of slopes $<-\mathrm{log}|\frac{d}{dx}|_{F_\rho}=\mathrm{log}\rho$.
	\end{definition}
	
	\begin{prop}
		$\mathrm{(}$\cite{Ked22}$\mathrm{,~Corollary~6.5.4)}$ Let $\ell$ be a non-constant, monic twisted polynomial in $F_\rho\{T\}$ and set $\omega=p^{-\frac{1}{p-1}}$. Then the multi-set of subsidiary radii of $F_\rho\{T\}/F_\rho\{T\}\ell$ consisting of radii $<\omega\rho$ is equal to the multi-set
		\[\{\omega\mathrm{exp}(\lambda):\lambda \ \mathrm{runs} \ \mathrm{through} \ \mathrm{ES}_\rho(\ell)\}.\]
	\end{prop}
	
	\subsection{Variation of generic and subsidiary radii}
	
	\begin{definition}
		Let $\theta,\gamma\in(0,1)$.
		
		(1) The \textit{ring of analytic functions} on the interval $[0,\gamma]$ (resp. $[0,1)$) is defined as
		\[\mathcal{A}_\gamma=\Big\{\sum_{i=0}^{\infty}a_ix^i~\Big|~a_i\in k,\lim_{i\rightarrow\infty}|a_i|\gamma^i=0\Big\}.\]
		\[\mathrm{(resp.}\quad\mathcal{A}=\Big\{\sum_{i=0}^{\infty}a_ix^i~\Big|~a_i\in k,\lim_{i\rightarrow\infty}|a_i|\eta^i=0,\forall~\eta\in(0,1)\Big\}\mathrm{)}.\]
		Let $\mathcal{A}_{\gamma,x}=\mathcal{A}_\gamma[x^{-1}]$, $\mathcal{A}_x=\mathcal{A}[x^{-1}]$ be the localization of $\mathcal{A}_\gamma$, $\mathcal{A}$ at $x$ respectively.
		
		(2) The \textit{ring of analytic functions} on the interval $[\theta,1)$ is defined as
		\[\mathcal{A}_{[\theta,1)}=\Big\{\sum_{i\in\mathbb{Z}}a_ix^i~\Big|~a_i\in k,\lim_{i\rightarrow-\infty}|a_i|\theta^i=0;\lim_{i\rightarrow\infty}|a_i|\eta^i=0,\forall~\eta\in(0,1)\Big\}.\]
		
		(3) The \textit{Robba ring} over $k$ is defined as
		\[\mathcal{R}=\lim_{\theta\rightarrow1^-}\mathcal{A}_{[\theta,1)}=\Big\{\sum_{i\in\mathbb{Z}} a_ix^i~\Big|~\lim_{i\rightarrow -\infty}|a_i|\xi^i=0,\exists~\xi\in(0,1);\lim_{i\rightarrow\infty}|a_i|\eta^i=0,\forall~\eta\in(0,1)\Big\}.\]
	\end{definition}
	
	\begin{rem}
		Let $S$ be one of the following rings: $\mathcal{A},\mathcal{A}_x,\mathcal{A}_\gamma,\mathcal{A}_{\gamma,x},\mathcal{A}_{[\theta,1)}$. For all appropriate values of $\rho$ in the table below, we have $S\subseteq F_\rho$, and the $\rho$-Gauss norm on $S$ is exactly
		\begin{equation}
			\Big|\sum_{i\in\mathbb{Z}}a_ix^i\Big|_\rho=\sup_i\{|a_i|\rho^i\}.
		\end{equation}
		
		\begin{table}[htbp]
			\centering
			\begin{tabular}{|c|c|c|c|}
				\hline
				$S$ & $\mathcal{A},\mathcal{A}_x$ & $\mathcal{A}_\gamma,\mathcal{A}_{\gamma,x}$ & $\mathcal{A}_{[\theta,1)}$ \\ 
				\hline
				$\rho$   & $\rho\in(0,1)$   &  $\rho\in(0,\gamma]$  & $\rho\in[\theta,1)$  \\
				\hline
			\end{tabular}
		\end{table}
	\end{rem}
	
	\begin{definition}
		Let $\theta,\gamma\in(0,1)$ and let $S$ be one of the following rings: $\mathcal{A},\mathcal{A}_x,\mathcal{A}_\gamma,\mathcal{A}_{\gamma,x},\mathcal{A}_{[\theta,1)}$. Let $M$ be a free differential module of rank $n$ over $(S,\frac{d}{dx})$.
		
		(1) For all appropriate values of $\rho$ in the above table, let
		\[R_1(M,\rho)\leq\cdots\leq R_n(M,\rho)\]
		denote all subsidiary radii of $M_\rho=M\otimes_SF_\rho$ listed in the increasing order. 
		
		(2) Define the functions
		\[f_i(M,r)=-\mathrm{log}R_i(M,\mathrm{exp}(-r)),\quad F_i(M,r)=\sum_{k=1}^{i}f_k(M,r)\]
		for all $1\leq i\leq n$, where $\mathrm{log}$ denotes the natural logarithm.
	\end{definition}
	
	\subsection{$P$-adic slopes}
	
	\begin{definition}
		$\mathrm{(}$\cite{C-M00}$\mathrm{,~D\'efinition~4.1\mbox{-}1)}$ A free differential module $M$ of finite rank over $(\mathcal{A}_x,\frac{d}{dx})$ is called \textit{solvable} if $\lim\limits_{\rho\rightarrow1^{-}}R_1(M,\rho)=1$.
	\end{definition}
	
	\begin{prop}
		$\mathrm{(}$\cite{C-M00}$\mathrm{,~Th\'eor\`eme~4.2\mbox{-}1~\&~}$\cite{Ked22}$\mathrm{,~Lemma~12.6.2)}$ Let $M$ be a solvable differential module of rank $n$ over $(\mathcal{A}_x,\frac{d}{dx})$.
		
		(1) There exist non-negative rational numbers $\alpha_1\geq\cdots\geq\alpha_n$ and some $\epsilon\in(0,1)$ such that
		\begin{equation}
			R_i(M,\rho)=\rho^{1+\alpha_i},\quad\forall~1\leq i\leq n,\quad\forall~\rho\in(1-\epsilon,1).
		\end{equation}
		
		(2) The summation $\sum_{i=1}^{n}\alpha_i$ is an integer.
	\end{prop}
	
	\begin{definition}
		Let $M$ be a solvable differential module of rank $n$ over $(\mathcal{A}_x,\frac{d}{dx})$.
		
		(1) The non-negative rational numbers $\alpha_1\geq\cdots\geq\alpha_n$ in (4.12.1) are called the $p$-\textit{adic slopes} of $M$.
		
		(2) The \textit{irregularity} of $M$ is an integer defined by
		\[\mathrm{Irr}(M)=\sum_{i=1}^{n}\alpha_i.\]
		
		(3) A solvable module is called \textit{pure of slope} $\alpha$ if all of its $p$-adic slopes are equal to $\alpha$.
	\end{definition}
	
	\begin{theorem}
		$\mathrm{(}$\cite{C-M00}$\mathrm{,~Corollaire~8.3\mbox{-}10,~Slope~Decomposition)}$ Let $M$ be a solvable differential module over $\mathcal{A}_x$. Then there exists a unique decomposition
		\[M\simeq\bigoplus\limits_{i=1}^{m}M(\alpha_i),\]
		where $\alpha_1>\cdots>\alpha_m\geq0$ and the direct summand $M(\alpha_i)$ is pure of slope $\alpha_i$. 
	\end{theorem}
	
	\begin{prop}
		Let $M$ be a solvable differential module of rank $n$ over $(\mathcal{A}_x,\frac{d}{dx})$, then
		
		(1) The functions $f_i(M,r)$ and $F_i(M,r)$ are continuous and piecewise affine on the interval $(0,\infty)$ for all $1\leq i\leq n$.
		
		(2) The function $F_i(M,r)$ is convex on the interval $(0,\infty)$ for all $1\leq i\leq n$.
	\end{prop}
	
	\begin{proof}
		According to Proposition 4.12, there exist non-negative rational numbers $\alpha_1\geq\cdots\geq\alpha_n$ and some $\epsilon\in(0,1)$ such that (4.12.1) holds.
		
		Fix some $\gamma\in(1-\epsilon,1)$. According to \cite{Ked22} Theorem 11.3.2, for all $1\leq i\leq n$ and on the interval $[-\mathrm{log}\gamma,\infty)$, $f_i(M,r)$ and $F_i(M,r)$ are continuous and piecewise affine, and $F_i(M,r)$ is convex.
		
		Meanwhile, (4.12.1) shows that on the interval $(0,-\mathrm{log}(1-\epsilon))$, we have
		\[f_i(M,r)=(1+\alpha_i)r,\quad F_i(M,r)=\Big(i+\sum_{j=1}^{i}\alpha_j\Big)r,\quad\forall~1\leq i\leq n.\]
		Therefore, $f_i(M,r)$ and $F_i(M,r)$ are continuous and piecewise affine, and $F_i(M,r)$ is convex.
		
		We conclude the proof by showing that the intersection $(0,-\mathrm{log}(1-\epsilon))\cap[-\mathrm{log}\gamma,\infty)$ is nonempty, which is evident since $\gamma>1-\epsilon$. 
	\end{proof}
	
	\section{Proof of theorem 1.1}
	
	\subsection{Discussion of cyclic basis}
	
	Our proof of Theorem 1.1 relies on the existence of a cyclic basis, so that we can work with Newton polygons. However, $M$ may not have a cyclic basis since $\mathcal{A}_x$ is not a field. Instead, we consider the ring $\mathcal{A}_{\gamma,x}$ in Definition 4.8. For the later proof in Section 5.3 and 5.4, it suffices to show that $\mathcal{A}_{\gamma,x}\otimes_{\mathcal{A}_x}M$ has a cyclic basis.
	
	\begin{lemma}
		Let $(M,D)$ be a free differential module of rank $n$ over $(\mathcal{A}_x,\frac{d}{dx})$. Then there exists $\gamma\in(0,1)$ such that $\mathcal{A}_{\gamma,x}\otimes_{\mathcal{A}_x}M$ has a cyclic basis.
	\end{lemma}
	
	\begin{proof}
		Let $e_1,\cdots,e_n$ be a basis of $M$. Let $\mathrm{Frac}(\mathcal{A}_x)$ denote the fraction field of $\mathcal{A}_x$.
		
		According to Theorem 2.3, $\mathrm{Frac}(\mathcal{A}_x)\otimes_{\mathcal{A}_x}M$ has a cyclic vector $m$. We have
		\begin{equation}
			(m,D(m),\cdots,D^{n-1}(m))=(e_1,\cdots,e_n)G,
		\end{equation}
		where the matrix $G=(g_{ij})_{n\times n}$ is invertible over $\mathrm{Frac}(\mathcal{A}_x)$.
		
		Each $g_{ij}$ can be uniquely written in the form $g_{ij}=x^{u_{ij}}\frac{s_{ij}}{t_{ij}}$, where $u_{ij}\in\mathbb{Z},~s_{ij},t_{ij}\in\mathcal{A}$ and the constant terms of $s_{ij},t_{ij}$ are nonzero.
		
		Similarly, $\mathrm{det}(G)$ can be uniquely written in the form $\mathrm{det}(G)=x^u\frac{s}{t}$, where $u\in\mathbb{Z},~s,t\in\mathcal{A}$ and the constant terms of $s,t$ are nonzero.
		
		Fix an algebraic closure $\overline{k}$ of $k$ and view $t_{ij}$ as a continuous function on the open unit disc $\{a\in\overline{k}:|a|<1\}$. Since $t_{ij}(0)\neq0$, there exists $\gamma_{ij}\in(0,1)$ such that $t_{ij}$ is nonzero on the neighbourhood $\{a\in\overline{k}:|a|\leq\gamma_{ij}\}$ of $0$. Therefore, $t_{ij}$ is invertible in $\mathcal{A}_{\gamma_{ij}}$.
		
		Similarly, there exists $\gamma_0\in(0,1)$ such that $s$ is invertible in $\mathcal{A}_{\gamma_0}$.
		
		Let
		\[\gamma'=\min_{1\leq i,j\leq n}\{\gamma_{ij}\},\quad\gamma=\min\{\gamma',\gamma_0\}.\]
		Now all of the $t_{ij}(1\leq i,j\leq n)$ and $s$ are invertible in $\mathcal{A}_\gamma$. Consequently, $G$ is invertible in $\mathcal{A}_{\gamma,x}$.
		
		We can therefore conclude from (5.1.1) that $\mathcal{A}_{\gamma,x}\otimes_{\mathcal{A}_x}M$ has a cyclic basis.
	\end{proof}
	
	\subsection{Computation of the $\rho$-Gauss norm}
	
	Fix some $\gamma\in(0,1)$. Let $\rho\in(0,\gamma]$, then $\mathcal{A}_{\gamma,x}$ is a subring of $F_\rho$ equipped with the $\rho$-Gauss norm $|\cdot|_\rho$. We establish a formula for computing the $\rho$-Gauss norm of a nonzero element of $\mathcal{A}_{\gamma,x}$ when $\rho$ is sufficiently small. This is the key to our proof of Theorem 1.1.
	
	\begin{lemma}
		Let $f=\sum_{i=m}^{\infty}b_ix^i$ be a nonzero element of $\mathcal{A}_{\gamma,x}$, where $b_i\in k,b_m\neq0$. There exists a constant $c>0$ such that
		\[|f|_\rho=|b_m|\rho^m=|b_m|\rho^{\mathrm{ord}_x(f)},\quad\forall\rho\in(0,c].\]
	\end{lemma}
	
	\begin{proof}
		Recall that, by (4.9.1), $|f|_\rho=\sup\limits_{i\geq m}\{|b_i|\rho^i\}$.
		
		Firstly, consider the $\frac{\gamma}{2}$-Gauss norm of $f$. We have 
		\[\lim\limits_{i\rightarrow\infty}|b_i|(\frac{\gamma}{2})^i\leq\lim\limits_{i\rightarrow\infty}|b_i|\gamma^i=0\]
		by definition of $\mathcal{A}_{\gamma,x}$, hence $\sup\limits_{i\geq m}\{|b_i|(\frac{\gamma}{2})^i\}<\infty$. Let $k$ be the largest integer such that
		\[|f|_\frac{\gamma}{2}=\sup_{i\geq m}\{|b_i|(\frac{\gamma}{2})^i\}=|b_k|(\frac{\gamma}{2})^k,\]
		then
		\begin{equation}
			|b_k|(\frac{\gamma}{2})^k>|b_i|(\frac{\gamma}{2})^i,\quad\forall~i>k.
		\end{equation}
		
		Pick a constant $c'>0$ such that
		\begin{equation}
			\rho^{m-i}>\Big|\frac{b_i}{b_m}\Big|,\quad\forall~m<i\leq k
		\end{equation}
		for all $\rho\in(0,c']$. We show that the desired constant is $c=\min\{c',\frac{\gamma}{2}\}$.
		
		For any $\rho\in(0,c]$, when $m<i\leq k$, we find that $|b_m|\rho^m>|b_i|\rho^i$ by our construction (5.2.2). When $i>k$, we infer from (5.2.1) that
		\[\Big|\frac{b_i}{b_k}\Big|<(\frac{\gamma}{2})^{k-i}\leq\rho^{k-i},\]
		hence $|b_m|\rho^m>|b_k|\rho^k>|b_i|\rho^i$.
		
		We conclude that $|b_m|\rho^m>|b_i|\rho^i$ for all $i>m$, i.e. $|f|_\rho=\sup\limits_{i\geq m}\{|b_i|\rho^i\}=|b_m|\rho^m$.
	\end{proof}
	
	\subsection{Analysis of Newton polygons}
	
	\begin{nota}
		In this subsection, we fix a real number $\gamma\in(0,1)$ and a non-constant, monic twisted polynomial $\ell=\sum_{i=0}^{n}a_iT^i\in\mathcal{A}_{\gamma,x}\{T\}$, where $a_i\in\mathcal{A}_{\gamma,x}$ and $a_n=1$. Write $a_i$ as a Laurent formal power series for $i=0,1\cdots,n$ as follows.
		\[a_i=b_ix^{\mathrm{ord}_x(a_i)}+\mathrm{higher} \ \mathrm{degree} \ \mathrm{terms},\quad b_i\in k^\times.\]
		
		Since $\mathcal{A}_{\gamma,x}$ can be naturally embedded into $k((x))$, we can define the formal slopes of $\mathcal{A}_{\gamma,x}\{T\}/\mathcal{A}_{\gamma,x}\{T\}\ell$ and write them in the decreasing order $\beta_1\geq\cdots\geq\beta_n$.
	\end{nota}
	
	\begin{rem}
		For any $\rho\in(0,\gamma]$, $\ell$ can be viewed as a twisted polynomial in $F_\rho\{T\}$ via the inclusion $\mathcal{A}_{\gamma,x}\subseteq F_\rho$. We can therefore talk about Newton polygons $\mathrm{NP}_\rho(\ell)$ in Definition 4.6.
		
		The main result of this subsection is Proposition 5.6, in which we compute the formal slopes of $\mathcal{A}_{\gamma,x}\{T\}/\mathcal{A}_{\gamma,x}\{T\}\ell$ via $\mathrm{NP}_\rho(\ell)$ when $\rho$ is sufficiently small. For comparison, see also Lemma 3.4, where we compute the formal slopes of $\mathcal{A}_{\gamma,x}\{T\}/\mathcal{A}_{\gamma,x}\{T\}\ell$ via $\mathrm{FNP}(\ell)$.
	\end{rem}
	
	\begin{con}
		According to Lemma 5.2, we can find a constant $c_1>0$ such that
		\begin{equation}
			|a_i|_\rho=|b_i|\rho^{\mathrm{ord}_x(a_i)},\quad i=0,1,\cdots,n
		\end{equation}
		for all $\rho\in(0,c_1]$.
		
		By solving a finite number of simple inequalities involving the variable $\rho$, we can find a constant $0<c_2<1$ such that
		\begin{equation}
			-\frac{1}{n^2+1}<\frac{\mathrm{log}\Big(\Big|\frac{b_j}{b_h}\Big|^\frac{1}{j-h}\cdot\Big|\frac{b_h}{b_k}\Big|^\frac{1}{k-h}\Big)}{-\mathrm{log}\rho}<\frac{1}{n^2+1}
		\end{equation}
		for all triples of integers $(h,j,k)$ with $0\leq h<j\leq k\leq n$ and all $\rho\in(0,c_2]$. Note that the numerator of the middle term is just a constant determined by $l$.
		
		For the same reason, we can find a constant $0<c_3<1$ such that
		\begin{equation}
			-\frac{1}{n^2+1}<\frac{\mathrm{log}\Big(\Big|\frac{b_{k'}}{b_{j'}}\Big|^\frac{1}{k'-j'}\cdot\Big|\frac{b_{h'}}{b_{j'}}\Big|^\frac{1}{h'-j'}\Big)}{-\mathrm{log}\rho}<\frac{1}{n^2+1}
		\end{equation}
		for all triples of integers $(h',j',k')$ with $0\leq h'<j'<k'\leq n$ and all $\rho\in(0,c_3]$.
		
		Finally, we can find a constant $0<c_4<1$ such that
		\begin{equation}
			-\frac{1}{n+1}<\frac{\mathrm{log}\Big|\frac{b_{k''}}{b_{j''}}\Big|^\frac{1}{k''-j''}}{-\mathrm{log}\rho}<\frac{1}{n+1}
		\end{equation}
		for all pairs of integers $(j'',k'')$ with $0\leq j''<k''\leq n$ and all $\rho\in(0,c_4]$.
		
		Take $c=\min\{c_1,c_2,c_3,c_4\}$.
	\end{con}
	
	Recall that in Definition 4.6 (4), effective slopes of $\mathrm{NP}_\rho(\ell)$ are slopes $<\mathrm{log}\rho$.
	
	\begin{prop}
		Let $c>0$ be the constant in Construction 5.5 and $-n=m_0<m_1\cdots<m_k=0$ be all the breaks of $\mathrm{NP}_c(\ell)$. Let $t$ be the largest integer such that the slope of $\mathrm{NP}_c(\ell)$ on the interval $[m_t,m_{t+1}]$ is effective. Then for any integer $s\in\{0,1,\cdots,t\}$, we have
		\[\beta_{m_s-m_0+1}=\beta_{m_s-m_0+2}=\cdots=\beta_{m_{s+1}-m_0}=-1-\frac{\mathrm{ord}_x(a_{-m_{s+1}})-\mathrm{ord}_x(a_{-m_s})}{m_{s+1}-m_s}.\]
	\end{prop}
	
	In order to prove Proposition 5.6, we need to establish several lemmas first.
	
	\begin{lemma}
		Let $c>0$ be the constant in Construction 5.5. Let $-n=m_0<m_1\cdots<m_k=0$ denote all the breaks of $\mathrm{NP}_c(\ell)$. Then for any $\rho\in(0,c]$, all the breaks of $\mathrm{NP}_\rho(\ell)$ are $-n=m_0<m_1\cdots<m_k=0$ as well.
	\end{lemma}
	
	\begin{proof}
		First we claim that, if for some integer $s\in\{0,1,\cdots,k-1\}$, $m_s$ is a break of $\mathrm{NP}_\rho(\ell)$ for all $\rho\in(0,c]$, then the smallest break of $\mathrm{NP}_\rho(\ell)$ strictly larger than $m_s$ is $m_{s+1}$ for all $\rho\in(0,c]$.
		
		Recall that in Definition 4.6 (2), $m_0=-n$ is a break of $\mathrm{NP}_\rho(\ell)$ for all $\rho\in(0,c]$. Therefore, by induction and the above claim, the proof is complete.
		
		We now prove the claim. For any $\rho\in(0,c]$, $m_s$ is already a break of $\mathrm{NP}_\rho(\ell)$ by our hypothesis in the claim. Then $m_{s+1}$ is the smallest break of $\mathrm{NP}_\rho(\ell)$ strictly larger than $m_s$ if and only if the following two conditions hold (see Figure 3).
		
		\[\frac{v_\rho(a_{-j})-v_\rho(a_{-m_{s+1}})}{j-m_{s+1}}\leq\frac{v_\rho(a_{-m_s})-v_\rho(a_{-m_{s+1}})}{m_s-m_{s+1}},\quad\forall~m_s\leq j<m_{s+1}.\]
		\[\frac{v_\rho(a_{-j})-v_\rho(a_{-m_{s+1}})}{j-m_{s+1}}>\frac{v_\rho(a_{-m_s})-v_\rho(a_{-m_{s+1}})}{m_s-m_{s+1}},\quad\forall~m_{s+1}<j\leq0.\]
		
		\begin{center}
			\begin{figure}[htbp]
				\begin{tikzpicture}
					\draw[->] (-4,0) -- (0.5,0) node[right] {$x$};
					\draw[->] (0,-2.5) -- (0,0.5) node[above] {$y$};
					\node[below left] at (0,0) {$0$};
					\node[above] at (-3.3,0) {$m_s$};
					\node[above] at (-1.5,0) {$m_{s+1}$};
					\filldraw[black] (-3.3,-0.8) circle (1pt); 
					\filldraw[black] (-2.5,-1) circle (1pt); 
					\filldraw[black] (-1.5,-2) circle (1pt); 
					\filldraw[black] (-0.9,-1.8) circle (1pt);
					\draw[black] (-3.5,-0.2) -- (-3.3,-0.8) -- (-1.5,-2) -- (-0.5,-2.2);
					\draw[loosely dashed] (-3.3,0) -- (-3.3,-0.8);
					\draw[loosely dashed] (-1.5,0) -- (-1.5,-2);
					\draw[loosely dashed] (-1.5,-2) -- (-2.5,-1);
					\draw[loosely dashed] (-1.5,-2) -- (-0.9,-1.8);
				\end{tikzpicture}
				\caption{}
			\end{figure}
		\end{center}
		
		Apply (5.5.1) to compute $v_\rho(a_{-j}),v_\rho(a_{-m_s}),v_\rho(a_{-m_{s+1}})$, it is equivalent to saying that the following two conditions hold.
		\begin{equation}
			\frac{\mathrm{ord}_x(a_{-m_{s+1}})-\mathrm{ord}_x(a_{-m_s})}{m_{s+1}-m_s}-\frac{\mathrm{ord}_x(a_{-m_{s+1}})-\mathrm{ord}_x(a_{-j})}{m_{s+1}-j}\geq\frac{\theta_1(-m_{s+1},-j,-m_s)}{-\mathrm{log}\rho},\quad\forall~m_s\leq j<m_{s+1}.
		\end{equation}
		\begin{equation}
			\frac{\mathrm{ord}_x(a_{-j})-\mathrm{ord}_x(a_{-m_{s+1}})}{j-m_{s+1}}-\frac{\mathrm{ord}_x(a_{-m_{s+1}})-\mathrm{ord}_x(a_{-m_s})}{m_{s+1}-m_s}>\frac{\theta_2(-j,-m_{s+1},-m_s)}{-\mathrm{log}\rho},\quad\forall~m_{s+1}<j\leq0.
		\end{equation}
		The junk term \[\theta_1(-m_{s+1},-j,-m_s)=\mathrm{log}\Big(\Big|\frac{b_{-j}}{b_{-m_{s+1}}}\Big|^\frac{1}{m_{s+1}-j}\cdot\Big|\frac{b_{-m_{s+1}}}{b_{-m_s}}\Big|^\frac{1}{m_{s+1}-m_s}\Big),\]
		\[\theta_2(-j,-m_{s+1},-m_s)=\mathrm{log}\Big(\Big|\frac{b_{-m_s}}{b_{-m_{s+1}}}\Big|^\frac{1}{m_{s+1}-m_s}\cdot\Big|\frac{b_{-j}}{b_{-m_{s+1}}}\Big|^\frac{1}{j-m_{s+1}}\Big).\]
		
		To prove the claim, it suffices to show that (5.7.1) and (5.7.2) hold for all $\rho\in(0,c]$.
		
		Note that (5.7.1) and (5.7.2) automatically hold for $\rho=c$ since by our initial assumption, $m_{s+1}$ is the smallest break of $\mathrm{NP}_c(\ell)$ strictly larger than $m_s$.
		
		For any $m_s\leq j<m_{s+1}$, let $A_s$ denote the left-hand side of (5.7.1). If $A_s<0$, then its absolute value $|A_s|_\mathbb{R}>\frac{1}{n^2}$ since the numerators and denominators are both integers, and the absolute values of the denominators are at most $n$. However, by (5.5.2) we have $\frac{\theta_1(-m_{s+1},-j,-m_s)}{-\mathrm{log}c}>-\frac{1}{n^2+1}$, which yields
		\[A_s<-\frac{1}{n^2}<-\frac{1}{n^2+1}<\frac{\theta_1(-m_{s+1},-j,-m_s)}{-\mathrm{log}c}.\]
	    This contradicts (5.7.1) when taking $\rho=c$. Therefore, $A_s\geq0$.
		
		If $\theta_1(-m_{s+1},-j,-m_s)<0$, then $A_s\geq0>\frac{\theta_1(-m_{s+1},-j,-m_s)}{-\mathrm{log}\rho}$ for all $\rho\in(0,c]$.
		
		If $\theta_1(-m_{s+1},-j,-m_s)\geq0$, then $A_s\geq\frac{\theta_1(-m_{s+1},-j,-m_s)}{-\mathrm{log}c}\geq\frac{\theta_1(-m_{s+1},-j,-m_s)}{-\mathrm{log}\rho}$ for all $\rho\in(0,c]$, where the first inequality comes from the fact that (5.7.1) holds for $\rho=c$.
		
		We can therefore conclude that (5.7.1) holds for all $\rho\in(0,c]$.
		
		For any $m_{s+1}<j\leq0$, let $B_s$ denote the left-hand side of (5.7.2). Then $B_s\geq0$ for the same reason as above by applying (5.5.3).
		
		If $\theta_2(-j,-m_{s+1},-m_s)<0$, then $B_s\geq0>\frac{\theta_2(-j,-m_{s+1},-m_s)}{-\mathrm{log}\rho}$ for all $\rho\in(0,c]$.
		
		If $\theta_2(-j,-m_{s+1},-m_s)\geq0$, then $B_s>\frac{\theta_2(-j,-m_{s+1},-m_s)}{-\mathrm{log}c}\geq\frac{\theta_2(-j,-m_{s+1},-m_s)}{-\mathrm{log}\rho}$ for all $\rho\in(0,c]$, where the first inequality comes from the fact that (5.7.2) holds for $\rho=c$.
		
		We can therefore conclude that (5.7.2) holds for all $\rho\in(0,c]$. The proof is now complete.
	\end{proof}
	
	\begin{lemma}
		Let $c>0$ be the constant in Construction 5.5. Let $-n=m_0<m_1<\cdots<m_k=0$ denote all the breaks of $\mathrm{NP}_c(\ell)$. For any integer $s\in\{0,1,\cdots,k-1\}$, if the slope of $\mathrm{NP}_c(\ell)$ on the interval $[m_s,m_{s+1}]$ is effective, then
		
		(1) The following inequality holds. 
		\[-1-\frac{\mathrm{ord}_x(a_{-m_{s+1}})-\mathrm{ord}_x(a_{-m_s})}{m_{s+1}-m_s}\geq0.\]
		
		(2) The slope of $\mathrm{NP}_\rho(\ell)$ on the interval $[m_s,m_{s+1}]$ is effective for all $\rho\in(0,c]$.
	\end{lemma}
	
	\begin{proof}
		Recall that in Definition 4.6 (4), the slope of $\mathrm{NP}_\rho(\ell)$ on the interval $[m_s,m_{s+1}]$ is effective if
		\[\frac{v_\rho(a_{-m_{s+1}})-v_\rho(a_{-m_s})}{m_{s+1}-m_s}<\mathrm{log}\rho.\]
		
		Apply (5.5.1) to compute $v_\rho(a_{-m_s}),v_\rho(a_{-m_{s+1}})$, it is equivalent to saying that the following inequality holds.
		\begin{equation}
			-1-\frac{\mathrm{ord}_x(a_{-m_{s+1}})-\mathrm{ord}_x(a_{-m_s})}{m_{s+1}-m_s}>\frac{\mathrm{log}\Big|\frac{b_{-m_s}}{b_{-m_{s+1}}}\Big|^\frac{1}{m_{s+1}-m_s}}{-\mathrm{log}\rho}.
		\end{equation}
		
		By our initial assumption, (5.8.1) automatically holds for $\rho=c$.
		
		Let $C_s$ denote the left-hand side of (5.8.1). If $C_s<0$, then its absolute value $|C_s|_\mathbb{R}\geq\frac{1}{n}$ since the numerators and denominators are both integers, and the absolute value of the denominator is at most $n$. However, by (5.5.4) we have $\frac{\mathrm{log}\Big|\frac{b_{-m_s}}{b_{-m_{s+1}}}\Big|^\frac{1}{m_{s+1}-m_s}}{-\mathrm{log}c}>-\frac{1}{n+1}$, which yields
		\[C_s<-\frac{1}{n}<-\frac{1}{n+1}<\frac{\mathrm{log}\Big|\frac{b_{-m_s}}{b_{-m_{s+1}}}\Big|^\frac{1}{m_{s+1}-m_s}}{-\mathrm{log}c}.\]
		This contradicts (5.8.1) when taking $\rho=c$. Therefore, $C_s\geq0$, and the statement (1) is proved.
		
		If $\mathrm{log}\Big|\frac{b_{-m_s}}{b_{-m_{s+1}}}\Big|^\frac{1}{m_{s+1}-m_s}<0$, then $C_s\geq0>\frac{\mathrm{log}\Big|\frac{b_{-m_s}}{b_{-m_{s+1}}}\Big|^\frac{1}{m_{s+1}-m_s}}{-\mathrm{log}\rho}$ for all $\rho\in(0,c]$.
		
		If $\mathrm{log}\Big|\frac{b_{-m_s}}{b_{-m_{s+1}}}\Big|^\frac{1}{m_{s+1}-m_s}\geq0$, then $C_s>\frac{\mathrm{log}\Big|\frac{b_{-m_s}}{b_{-m_{s+1}}}\Big|^\frac{1}{m_{s+1}-m_s}}{-\mathrm{log}c}\geq\frac{\mathrm{log}\Big|\frac{b_{-m_s}}{b_{-m_{s+1}}}\Big|^\frac{1}{m_{s+1}-m_s}}{-\mathrm{log}\rho}$ for all $\rho\in(0,c]$, where the first inequality comes from the fact that (5.8.1) holds for $\rho=c$.
		
		We can therefore conclude that (5.8.1) holds for all $\rho\in(0,c]$, and the statement (2) is proved.
	\end{proof}
	
	We are now ready to compare $\mathrm{NP}_\rho(\ell)$ with the formal Newton polygon $\mathrm{FNP}(\ell)$.
	
	\begin{lemma}
		Let $c>0$ be the constant in Construction 5.5. Let $-n=m_0<m_1\cdots<m_k=0$ denote all the breaks of $\mathrm{NP}_c(\ell)$.
		
		(1) The set of breaks of $\mathrm{FNP}(\ell)$ is contained in $\{m_0,m_1,\cdots,m_k\}$.
		
		(2) For any integer $s\in\{0,1,\cdots,k\}$, if $m_s$ is not a break of $\mathrm{FNP}(\ell)$, then $\mathrm{FNP}(\ell)$ passes through the point $(m_s,\mathrm{ord}_x(a_{-m_s}))\in\mathbb{R}^2$.
	\end{lemma}
	
	\begin{proof}
		The statement (1) can be replaced by an equivalent statement as follows.
		
		\textit{(1') For any integer} $s\in\{0,1,\cdots,k-1\}$\textit{, there is no break} $j$ \textit{of} $\mathrm{FNP}(\ell)$ \textit{such that} $m_s<j<m_{s+1}$\textit{.}
		
		We will prove (1') and (2) by induction on $s$.
		
		\textit{Step 1.}
		
		First we show that (1') and (2) hold for $s=0$.
		
		Since $m_0=-n$ is a break of $\mathrm{FNP}(\ell)$, (2) automatically holds for $s=0$.
		
		Let $u_0$ be the smallest break of $\mathrm{FNP}(\ell)$ strictly larger than $m_0$. If $m_0<u_0<m_1$, then
		\begin{equation}
			\frac{\mathrm{ord}_x(a_{-m_1})-\mathrm{ord}_x(a_{-u_0})}{m_1-u_0}>\frac{\mathrm{ord}_x(a_{-u_0})-\mathrm{ord}_x(a_{-m_0})}{u_0-m_0}.
		\end{equation}
		
		Note that $m_1$ is the smallest break of $\mathrm{NP}_c(\ell)$ strictly larger than $m_0$. As shown in Lemma 5.7, the left-hand side of (5.7.1) is $\geq 0$. Replace $m_s,j,m_{s+1}$ in (5.7.1) with $m_0,u_0,m_1$ respectively,
		\begin{equation}
			\frac{\mathrm{ord}_x(a_{-m_1})-\mathrm{ord}_x(a_{-m_0})}{m_1-m_0}\geq\frac{\mathrm{ord}_x(a_{-m_1})-\mathrm{ord}_x(a_{-u_0})}{m_1-u_0}.
		\end{equation}
		
		Let $G_0$ (resp. $H_0$) denote the left-hand (resp. right-hand) side of (5.9.2). We have
		\[\frac{\mathrm{ord}_x(a_{-u_0})-\mathrm{ord}_x(a_{-m_0})}{u_0-m_0}-H_0=\frac{(m_1-m_0)G_0-(m_1-u_0)H_0}{u_0-m_0}-H_0=\frac{(m_1-m_0)(G_0-H_0)}{u_0-m_0}\geq0.\]
		This contradicts (5.9.1), so $u_0\geq m_1$ and (1') holds for $s=0$.
		
		\textit{Step 2.}
		
		Suppose that (1') and (2) hold up to $s-1$. We will then show that (2) also holds for $s$. 
		
		If $m_s$ is a break of $\mathrm{FNP}(\ell)$, then (2) automatically holds for $s$.
		
		Otherwise let $u_s>m_s$ be the smallest break of $\mathrm{FNP}(\ell)$ strictly larger than $m_s$.
		
		By the induction hypothesis, we can find an integer $t\leq s-1$ such that $m_t$ is the largest break of $\mathrm{FNP}(\ell)$ strictly smaller than $m_s$, and $\mathrm{FNP}(\ell)$ passes through the point $(m_{s-1},\mathrm{ord}_x(a_{-m_{s-1}}))\in\mathbb{R}^2$ (see Figure 4).
		
		\begin{center}
			\begin{figure}[htbp]
				\begin{tikzpicture}
					\draw[->] (-4,0) -- (0.5,0) node[right] {$x$};
					\draw[->] (0,-2.5) -- (0,0.5) node[above] {$y$};
					\node[below left] at (0,0) {$0$};
					\node[above] at (-3.3,0) {$m_t$};
					\node[above] at (-2.5,0) {$m_{s-1}$};
					\node[above] at (-1.7,0) {$m_s$};
					\node[above] at (-1.3,0) {$u_s$};
					\filldraw[black] (-3.3,-0.6) circle (1pt); 
					\filldraw[black] (-2.5,-1.2) circle (1pt); 
					\filldraw[black] (-1.7,-1.8) circle (1pt); 
					\filldraw[black] (-1.3,-2.1) circle (1pt);
					\draw[black] (-3.5,-0.2) -- (-3.3,-0.6) -- (-1.3,-2.1) -- (-0.5,-2.2);
					\draw[loosely dashed] (-3.3,0) -- (-3.3,-0.6);
					\draw[loosely dashed] (-2.5,0) -- (-2.5,-1.2);
					\draw[loosely dashed] (-1.7,0) -- (-1.7,-1.8);
					\draw[loosely dashed] (-1.3,0) -- (-1.3,-2.1);
				\end{tikzpicture}
				\caption{}
			\end{figure}
		\end{center}
		
		Note that $\mathrm{FNP}(\ell)$ on the interval $[m_t,u_s]$ is a segment with endpoints at $(m_t,\mathrm{ord}_x(a_{-m_t}))$ and $(u_s,\mathrm{ord}_x(a_{-u_s}))$. The segment passes through $(m_{s-1},\mathrm{ord}_x(a_{-m_{s-1}}))$, i.e.
		\begin{equation}
			\frac{\mathrm{ord}_x(a_{-u_s})-\mathrm{ord}_x(a_{-m_t})}{u_s-m_t}=\frac{\mathrm{ord}_x(a_{-u_s})-\mathrm{ord}_x(a_{-m_{s-1}})}{u_s-m_{s-1}}.
		\end{equation}
		
		To show that (2) is true for $s$, it suffices to prove that the segment with endpoints at $(m_{s-1},\mathrm{ord}_x(a_{-m_{s-1}}))$ and $(u_s,\mathrm{ord}_x(a_{-u_s}))$ passes through the point $(m_s,\mathrm{ord}_x(a_{-m_s}))$, i.e. the following equality holds.
		\begin{equation}
			\frac{\mathrm{ord}_x(a_{-u_s})-\mathrm{ord}_x(a_{-m_{s-1}})}{u_s-m_{s-1}}=\frac{\mathrm{ord}_x(a_{-u_s})-\mathrm{ord}_x(a_{-m_s})}{u_s-m_s}
		\end{equation}
		
		Since $m_t$ is a break of $\mathrm{FNP}(\ell)$ and $u_s$ is the smallest break of $\mathrm{FNP}(\ell)$ strictly larger than $m_t$,
		\begin{equation}
			\frac{\mathrm{ord}_x(a_{-u_s})-\mathrm{ord}_x(a_{-m_{s-1}})}{u_s-m_{s-1}}=\frac{\mathrm{ord}_x(a_{-u_s})-\mathrm{ord}_x(a_{-m_t})}{u_s-m_t}\geq\frac{\mathrm{ord}_x(a_{-u_s})-\mathrm{ord}_x(a_{-m_s})}{u_s-m_s},
		\end{equation}
		where the first equality comes from (5.9.3).
		
		Note that $m_s$ is the smallest break of $\mathrm{NP}_c(\ell)$ strictly larger than $m_{s-1}$. As shown in Lemma 5.7, the left-hand side of (5.7.2) is $\geq 0$. Replace $m_s,m_{s+1},j$ in (5.7.2) with $m_{s-1},m_s,u_s$ respectively,
		\begin{equation}
			\frac{\mathrm{ord}_x(a_{-u_s})-\mathrm{ord}_x(a_{-m_s})}{u_s-m_s}\geq\frac{\mathrm{ord}_x(a_{-m_s})-\mathrm{ord}_x(a_{-m_{s-1}})}{m_s-m_{s-1}}.
		\end{equation}
		
		Let $U_s$ (resp. $V_s$) be the left-hand (resp. right-hand) side of (5.9.6). We have
		\begin{equation}
			\frac{\mathrm{ord}_x(a_{-u_s})-\mathrm{ord}_x(a_{-m_{s-1}})}{u_s-m_{s-1}}-U_s=\frac{(u_s-m_s)U_s+(m_s-m_{s-1})V_s}{u_s-m_{s-1}}-U_s=\frac{(m_s-m_{s-1})(V_s-U_s)}{u_s-m_{s-1}}\leq0.
		\end{equation}
		
		When we combine (5.9.5) and (5.9.7), we see that (5.9.4) holds. Therefore, (2) is true for $s$.
		
		\textit{Step 3.}
		
		Suppose that (1') holds up to $s-1$ and (2) holds up to $s$. We will then show that (1') also holds for $s$.
		
		If $m_s$ is a break of $\mathrm{FNP}(\ell)$, we can repeat the process in Step 1 to justify (1') for $s$.
		
		Otherwise let $m_t$ and $u_s$ have the same meaning as in Step 2. Then $\mathrm{FNP}(\ell)$ on the interval $[m_t,u_s]$ is a segment with endpoints at $(m_t,\mathrm{ord}_x(a_{-m_t}))$ and $(u_s,\mathrm{ord}_x(a_{-u_s}))$, and the segment passes through the point $(m_s,\mathrm{ord}_x(a_{-m_s}))$.
		
		If $m_s<u_s<m_{s+1}$, apply the same technique as in Step 2 to $\mathrm{FNP}(\ell)$,
		\begin{equation}
			\frac{\mathrm{ord}_x(a_{-m_{s+1}})-\mathrm{ord}_x(a_{-u_s})}{m_{s+1}-u_s}>\frac{\mathrm{ord}_x(a_{-u_s})-\mathrm{ord}_x(a_{-m_t})}{u_s-m_t}=\frac{\mathrm{ord}_x(a_{-u_s})-\mathrm{ord}_x(a_{-m_s})}{u_s-m_s}.
		\end{equation}
		
		Apply the same technique as in Step 2 again to $\mathrm{NP}_c(\ell)$,
		\begin{equation}
			\frac{\mathrm{ord}_x(a_{-m_{s+1}})-\mathrm{ord}_x(a_{-m_s})}{m_{s+1}-m_s}\geq\frac{\mathrm{ord}_x(a_{-m_{s+1}})-\mathrm{ord}_x(a_{-u_s})}{m_{s+1}-u_s}.
		\end{equation}
		
		Let $G_s$ (resp. $H_s$) be the left-hand (resp. right-hand) side of (5.9.9). We have
		\[\frac{\mathrm{ord}_x(a_{-u_s})-\mathrm{ord}_x(a_{-m_s})}{u_s-m_s}-H_s=\frac{(m_{s+1}-m_s)G_s-(m_{s+1}-u_s)H_s}{u_s-m_s}-H_s=\frac{(m_{s+1}-m_s)(G_s-H_s)}{u_s-m_s}\geq0.\]
		This contradicts (5.9.8), so $u_s\geq m_{s+1}$ and (1') holds for $s$.
		
		We have finished the induction.
	\end{proof}
	
	We are now ready to prove Proposition 5.6.
	
	\begin{proof}
		Recall that in Definition 3.3 (3), the multi-set of slopes of $\mathrm{FNP}(\ell)$ is denoted by $\mathrm{FS}(\ell)$.
		
		According to Lemma 5.9, $\mathrm{FS}(\ell)$ can be reformulated as the disjoint union
		\[\mathrm{FS}(\ell)=\mathop{\bigsqcup}_{s=0}^{k-1}\Big\{\frac{\mathrm{ord}_x(a_{-m_{s+1}})-\mathrm{ord}_x(a_{-m_s})}{m_{s+1}-m_s} \ \mathrm{with} \ \mathrm{multiplicity} \ m_{s+1}-m_s\Big\},\]
		where
		\[\frac{\mathrm{ord}_x(a_{-m_1})-\mathrm{ord}_x(a_{-m_0})}{m_1-m_0}\leq\frac{\mathrm{ord}_x(a_{-m_2})-\mathrm{ord}_x(a_{-m_1})}{m_2-m_1}\leq\cdots\leq\frac{\mathrm{ord}_x(a_{-m_k})-\mathrm{ord}_x(a_{-m_{k-1}})}{m_k-m_{k-1}}.\]
		
		The proof is completed using Lemma 3.4 and Lemma 5.8 (1).
	\end{proof}
	
	\subsection{Final proof of Theorem 1.1}
	\begin{nota}
		According to Lemma 5.1, we can pick $\gamma\in(0,1)$ such that $\mathcal{A}_{\gamma,x}\otimes_{\mathcal{A}_x}M$ has a cyclic basis. We can then assume that $\mathcal{A}_{\gamma,x}\otimes_{\mathcal{A}_x}M\simeq\mathcal{A}_{\gamma,x}\{T\}/\mathcal{A}_{\gamma,x}\{T\}\ell$ for some monic twisted polynomial $\ell=\sum_{i=0}^{n}a_iT^i$, where $a_i\in\mathcal{A}_{\gamma,x}$ and $a_n=1$. Note that the formal slopes of $M$ can be computed as those of $\mathcal{A}_{\gamma,x}\{T\}/\mathcal{A}_{\gamma,x}\{T\}\ell$ since $\mathcal{A}_x\subseteq\mathcal{A}_{\gamma,x}\subseteq k((x))$. Write
		\[a_i=b_ix^{\mathrm{ord}_x(a_i)}+\mathrm{higher} \ \mathrm{degree} \ \mathrm{terms},\quad b_i\in k^\times,\quad i=0,1,\cdots,n.\]
		Let $c$ be the constant in Construction 5.5. Let $-n=m_0<m_1\cdots<m_k=0$ denote all the breaks of $\mathrm{NP}_c(\ell)$.	
	\end{nota}
	
	\begin{con}
		Suppose that the largest $p$-adic slope $\alpha_1>0$. We can find a constant $c'$ such that $\rho^{\alpha_1}<\omega=p^{-\frac{1}{p-1}}$ for all $\rho\in(0,c']$.
		
		Take $C_1=\min\{c,c'\}$.
	\end{con}
	
	\begin{lemma}
		Suppose that $\alpha_1>0$. Let $C_1$ be the constant in Construction 5.11.
		
		(1) The inequality (1.1.1) holds for all $1\leq i\leq m_1-m_0$.
		
		(2) When $0<j\leq m_1-m_0$, the subsidiary radii can be computed as follows.
		\[R_j(M,\rho)=\omega\rho^{1+\beta_j}\Big|\frac{1}{b_{-m_1}}\Big|^\frac{1}{m_1-m_0},\quad\forall~\rho\in(0,C_1].\]
	\end{lemma}
	
	\begin{proof}
		Recall that by Proposition 4.12, there exists a small enough constant $\epsilon>0$ such that
		\begin{equation}
			0<C_1<1-\epsilon<1,\quad R_i(M,\rho)=\rho^{1+\alpha_i},\quad\forall~\rho\in(1-\epsilon,1) \ \mathrm{and} \ \forall~1\leq i\leq n.
		\end{equation}
		
		Since the function $F_1(M,r)=-\mathrm{log}R_1(M,\mathrm{exp}(-r))$ is continuous, piecewise affine and convex on the interval $r\in(0,\infty)$ by Proposition 4.15, when $\rho\in(0,C_1]$ we have
		\[R_1(M,\rho)\leq\rho^{1+\alpha_1}<\omega\rho,\]
		where the second inequality comes from Construction 5.11. By Lemma 5.7 and Proposition 4.7,
		\begin{equation}
			R_1(M,\rho)=\cdots=R_{m_1-m_0}(M,\rho)=\omega\mathrm{exp}(\frac{v_\rho(a_{-m_1})-v_\rho(a_{-m_0})}{m_1-m_0}),\quad\forall~\rho\in(0,C_1].
		\end{equation}
		
		Apply (5.5.1) to compute $v_\rho(a_{-m_1})$ and $v_\rho(a_{-m_0})$, and note that $b_{-m_0}=b_n=1$,
		\begin{equation}
			\frac{v_\rho(a_{-m_1})-v_\rho(a_{-m_0})}{m_1-m_0}=(-\mathrm{log}\rho)\cdot\frac{\mathrm{ord}_x(a_{-m_1})-\mathrm{ord}_x(a_{-m_0})}{m_1-m_0}+\mathrm{log}\Big|\frac{1}{b_{-m_1}}\Big|^\frac{1}{m_1-m_0}.
		\end{equation}
		
		The inequality $R_1(M,C_1)<\omega C_1$ forces the slope of $\mathrm{NP}_{C_1}(\ell)$ on the interval $[m_0,m_1]$ to be effective. Then by Proposition 5.6,
		\begin{equation}
			\beta_1=\beta_2=\cdots=\beta_{m_1-m_0}=-1-\frac{\mathrm{ord}_x(a_{-m_1})-\mathrm{ord}_x(a_{-m_0})}{m_1-m_0}.
		\end{equation}
		
		Combine (5.12.2), (5.12.3) and (5.12.4),
		\begin{equation}
			R_1(M,\rho)=\cdots=R_{m_1-m_0}(M,\rho)=\omega\rho^{1+\beta_1}\Big|\frac{1}{b_{-m_1}}\Big|^\frac{1}{m_1-m_0},\quad\forall~\rho\in(0,C_1].
		\end{equation}
		We have proved (2).
		
		Combining (5.12.1) and (5.12.5), we see that $F_1(M,r)$ is a straight line with a slope of $1+\alpha_1$ (resp. $1+\beta_1$) on the interval $(0,-\mathrm{log}(1-\epsilon))$ (resp. $[-\mathrm{log}C_1,\infty)$). The convexity of $F_1(M,r)$ on the interval $(0,\infty)$ then forces $1+\alpha_1\leq1+\beta_1$ (see Figure 1). Therefore, we have
		\[\beta_1=\cdots=\beta_{m_1-m_0}\geq\alpha_1\geq\cdots\geq\alpha_{m_1-m_0},\]
		and (1.1.1) holds for all $1\leq i\leq m_1-m_0$.
	\end{proof}
	
	Case-by-case proof of Theorem 1.1 is now ready to be presented.
	
	\begin{proof}
		First we may assume that $\alpha_1>0$. Otherwise $\alpha_1=\cdots=\alpha_n=0$ and (1.1.1) automatically holds, since formal slopes are non-negative.
		
		Next we may assume that $k\geq2$. Otherwise $k=1$ and $m_1-m_0=m_k-m_0=n$, in which case we can conclude the proof by Lemma 5.12 (1).
		
		Write $\sigma=\sum_{j=1}^{n}\alpha_j$. There are several cases.
		
		\textit{Case 1.} $\sigma\leq\sum_{j=1}^{m_1-m_0}\beta_j$.
		
		In this case, (1.1.1) holds for all $1\leq i\leq m_1-m_0$ by Lemma 5.12 (1). For any $m_1-m_0<i\leq n$, we have
		\[\sum_{j=1}^{i}\alpha_j\leq\sigma\leq\sum\limits_{j=1}^{m_1-m_0}\beta_j\leq\sum_{j=1}^{i}\beta_j.\]
		Therefore, (1.1.1) holds for all $1\leq i\leq n$ and the proof is complete.
		
		\textit{Case 2.} $\sigma>\sum_{j=1}^{m_1-m_0}\beta_j$.
		
		There are two subcases.
		
		\textit{Case 2-1.} $\sigma\leq\sum_{j=1}^{m_k-m_0}\beta_j=\sum_{j=1}^{n}\beta_j$.
		
		In this subcase, let $t\in\{1,2,\cdots,k-1\}$ be the largest integer such that $\sigma>\sum_{j=1}^{m_t-m_0}\beta_j$.
		
		According to Lemma 5.14 (1) below, (1.1.1) holds for all $1\leq i\leq m_{t+1}-m_0$.
		
		If $m_{t+1}-m_0=n$, then the result follows. Otherwise for any $m_{t+1}-m_0<i\leq n$, we have
		\[\sum_{j=1}^{i}\alpha_j\leq\sigma\leq\sum_{j=1}^{m_{t+1}-m_0}\beta_j\leq\sum_{j=1}^{i}\beta_j,\]
		where the middle inequality comes from the maximality of $t$.
		
		We therefore conclude that (1.1.1) holds for all $1\leq i\leq n$.
		
		\textit{Case 2-2.} $\sigma>\sum_{j=1}^{m_k-m_0}\beta_j=\sum_{j=1}^{n}\beta_j$.
		
		We show that this subcase does not occur. 
		
		Otherwise take $s=t'+1=k$ in Lemma 5.14 (1) below. We find that  $\sum_{j=1}^{m_k-m_0}\alpha_j\leq\sum_{j=1}^{m_k-m_0}\beta_j$. However, this contradicts the fact that $\sigma=\sum_{j=1}^{m_k-m_0}\alpha_j>\sum_{j=1}^{m_k-m_0}\beta_j$.
		
		We have considered all possible cases and completed the proof.	
	\end{proof}
	
	\begin{con}
		In \textit{Case 2}, let $t'=t$ if \textit{Case 2-1} occurs and $t'=k-1$ if \textit{Case 2-2} occurs.
		
		By solving a finite number of simple inequalities involving the variable $\rho$, we can find a constant $c''>0$ (which depends on $t'$) such that
		\begin{equation}
			\rho^\frac{\sigma-\sum_{j=1}^{n-v}\beta_j}{v}<\omega^\frac{n}{v}\Big|\frac{1}{b_v}\Big|^\frac{1}{v},\quad\forall~-m_{t'}\leq v<n
		\end{equation}
		for all $\rho\in(0,c'']$, where $\omega=p^{-\frac{1}{p-1}}$ and $b_0,b_1,\cdots,b_n$ are as in Notation 5.10.
		
		Let $C_1$ be the constant in Construction 5.11, take $C_2=\min\{C_1,c''\}$. 
	\end{con}
	
	\begin{lemma}
		Assuming the conditions of Case 2, let $t'$ and $C_2$ be as defined in Construction 5.13. Then for any integer $s\in\{1,2,\cdots,t'+1\}$, the following statements hold.
		
		(1) The inequality (1.1.1) holds for all $1\leq i\leq m_s-m_0$.
		
		(2) When $m_{s-1}-m_0<j\leq m_s-m_0$, the subsidiary radii can be computed as follows.
		\[R_j(M,\rho)=\omega\rho^{1+\beta_j}\Big|\frac{b_{-m_{s-1}}}{b_{-m_s}}\Big|^\frac{1}{m_s-m_{s-1}},\quad\forall~\rho\in(0,C_2].\]
	\end{lemma}
	
	\begin{proof}
		We will prove (1) and (2) by induction on $s$.
		
		According to Proposition 4.12, there exists a sufficiently small constant $\epsilon>0$ such that
		\begin{equation}
			0<C_2<1-\epsilon<1,\quad R_i(M,\rho)=\rho^{1+\alpha_i},\quad\forall~\rho\in(1-\epsilon,1) \ \mathrm{and} \ \forall~1\leq i\leq n.
		\end{equation}
		
		The case where $s=1$ has already been proved in Lemma 5.12.
		
		Now suppose that (1) and (2) hold up to $s$ for some $s\in\{1,\cdots,t'\}$. We will then show that (1) and (2) also hold for $s+1$.
		
		By combining the convexity and continuity of $F_n(M,r)$ on the interval $(0,\infty)$ in Proposition 4.15 with (5.14.1), we can deduce that $\prod_{j=1}^{n}R_j(M,\rho)\leq\rho^{n+\sigma}$ when $\rho\in(0,C_2]$.
		
		Since the subsidiary radii are listed in the increasing order $R_1(M,\rho)\leq\cdots\leq R_n(M,\rho)$, we find that
		\begin{equation}
			\Big(R_{m_s-m_0+1}(M,\rho)\Big)^{-m_s}\cdot\prod_{j=1}^{m_s-m_0}R_j(M,\rho)\leq\prod_{j=1}^{n}R_j(M,\rho)\leq\rho^{n+\sigma}.
		\end{equation}
		
		Using the induction hypothesis in (2) and noting that $b_{-m_0}=b_n=1$, we have
		
		\begin{equation}
			\prod_{j=1}^{m_s-m_0}R_j(M,\rho)=\omega^{m_s-m_0}\cdot\rho^{(m_s-m_0+\sum\limits_{j=1}^{m_s-m_0}\beta_j)}\cdot\Big|\frac{1}{b_{-m_s}}\Big|.
		\end{equation}
		
		Combining (5.14.2) and (5.14.3), we have
		\[R_{m_s-m_0+1}(M,\rho)\leq\rho^{1+\frac{\sigma-\sum_{j=1}^{m_s-m_0}\beta_j}{-m_s}}\omega^{1-\frac{m_0}{m_s}}|b_{-m_s}|^\frac{1}{-m_s}<\omega\rho,\]
		where the second inequality comes from (5.13.1).
		
		Since $R_{m_s-m_0+1}(M,\rho)<\omega\rho$, by Lemma 5.7 and Proposition 4.7, we have
		\[R_{m_s-m_0+1}(M,\rho)=\cdots=R_{m_{s+1}-m_0}(M,\rho)=\omega\mathrm{exp}(\frac{v_\rho(a_{-m_{s+1}})-v_\rho(a_{-m_s})}{m_{s+1}-m_s}),\quad\forall~\rho\in(0,C_2].\]
		
		Using (5.5.1) to calculate $v_\rho(a_{-m_{s+1}})$ and $v_\rho(a_{-m_s})$, then combining the results with the following equation in Proposition 5.6
		\[\beta_{m_s-m_0+1}=\cdots=\beta_{m_{s+1}-m_0}=-1-\frac{\mathrm{ord}_x(a_{-m_{s+1}})-\mathrm{ord}_x(a_{-m_s})}{m_{s+1}-m_s},\]
		we find that
		\begin{equation}
			R_j(M,\rho)=\omega\rho^{1+\beta_j}\Big|\frac{b_{-m_s}}{b_{-m_{s+1}}}\Big|^\frac{1}{m_{s+1}-m_s},\quad\forall~\rho\in(0,C_2] \ \mathrm{and} \ \forall~m_s-m_0<j\leq m_{s+1}-m_0.
		\end{equation}
		We have proved that (2) is true for $s+1$.
		
		The induction hypothesis in (2') up to $s$ and (5.14.4) show that $F_i(M,r)$ is a straight line with a slope of $i+\sum_{j=1}^{i}\beta_j$ on the interval $[-\mathrm{log}C_2,\infty)$ for all $m_s-m_0<i\leq m_{s+1}-m_0$. Meanwhile, (5.14.1) shows that $F_i(M,r)$ is a straight line with a slope of $i+\sum_{j=1}^{i}\alpha_j$ on the interval $(0,-\mathrm{log}(1-\epsilon))$. The convexity of $F_i(M,r)$ on the interval $(0,\infty)$ then forces $i+\sum_{j=1}^{i}\alpha_j\leq i+\sum_{j=1}^{i}\beta_j$ (see Figure 1). Therefore, (1.1.1) holds for all $m_s-m_0<i\leq m_{s+1}-m_0$.
		
		By the induction hypothesis, (1.1.1) already holds for all $1\leq i\leq m_s-m_0$. We can therefore conclude that (1.1.1) holds for all $1\leq i\leq m_{s+1}-m_0$ and that (1) also holds for $s+1$.
		
		We have finished the induction.
	\end{proof}
	
	\section{Examples of Bessel equations}
	
	Let $(k,|\cdot|)$ be a complete discrete valuation field of mixed characteristic $(0,p)$ with a finite residue field $\kappa$. Fix an algebraic closure $\overline{k}$ of $k$ and let $k^\mathrm{ur}$ be the maximum unramified extension of $k$ in $\overline{k}$. Then the residue field of $k^\mathrm{ur}$ is an algebraic closure of $\kappa$, which we denote by $\overline{\kappa}$.
	
	Suppose that $k$ contains an element $\pi$ such that $\pi^{p-1}=-p$. Normalize the absolute value so that $|p|=p^{-1}$, then $|\pi|=\omega=p^{-\frac{1}{p-1}}$.
	
	\subsection{Monodromy representations and Swan conductors}
	
	Let $F=\overline{\kappa}((x))$ and $I_F$ be the Galois group $\mathrm{Gal}(F^\mathrm{sep}/F)$. Given a \textit{quasi-unipotent} differential module $M$ of rank $n$ over $\mathcal{R}$ (see \cite{Ked05}, Section 4.6), we can associate to $M$ a representation
	\[\rho_M:I_F\rightarrow\mathrm{GL}_n(k^\mathrm{ur})\]
	together with a nilpotent operator $N\in\mathfrak{gl}_n(k^\mathrm{ur})$, such that the image of $\rho_M$ is finite and all elements in the image of $\rho_M$ commute with $N$ (see \cite{Ked05}, Section 4.7). The representation $\rho_M$ is called the \textit{monodromy representation} associated with $M$.
	
	Let $G$ be the image of $\rho_M$. There exists a unique decomposition of $\rho_M$ known as the \textit{break decomposition} (\cite{Ka88}, Proposition 1.1).
	\[\rho_M\simeq\bigoplus\limits_{\alpha}\rho_M(\alpha),\]
	where $\alpha$ runs through the set of jumps of the upper number ramification filtration of $G$ (see \cite{Se79}, Chapter IV.3).
	
	An index $\alpha$ such that $\rho_M(\alpha)\neq0$ is called a \textit{break} of $\rho_M$. The \textit{Swan conductor} of $\rho_M$ is defined by
	\[\mathrm{Swan}(\rho_M)=\sum_{\mathrm{breaks}~\alpha}\alpha\cdot\mathrm{rank}(\rho_M(\alpha)).\]
	
	Since quasi-unipotent modules are solvable (\cite{Ked22}, Proposition 20.1.3), it makes sense to talk about the $p$-adic slopes of $M$. Recall the irregularity of $M$ in Definition 4.13. According to a result of Tsuzuki (\cite{Tsu98}, Theorem 7.2.2), we have
	\[\mathrm{Irr}(M)=\mathrm{Swan}(\rho_M).\]
	
	Let $M'$ be an irreducible submodule of $M$, then $M'$ is of pure slope by Theorem 4.14. Applying Tsuzuki's result to all irreducible submodules $M'$, we deduce that the $p$-adic slopes of $M$ coincide with the breaks of $\rho_M$.
	
	We conclude the discussion above with the following corollary.
	
	\begin{coro}
		Let $M$ be a quasi-unipotent module over $\mathcal{R}$. Let $G$ be the image of the monodromy representation associated with $M$. Then the $p$-adic slopes of $M$ take value in the set of jumps of the upper number ramification filtration of $G$.
	\end{coro}

	\subsection{Bessel equations}
	
	Let $\mathcal{O}_k$ be the ring of integers in $k$. Let $\sigma:k\rightarrow k$ be a lifting of the $p$-th power Frobenius on $\kappa$ such that $\sigma(\pi)=\pi$. Let $X=\mathbb{P}_\kappa^1-\{0,\infty\}$ and $\hat{\mathbb{P}}^1$ be the $p$-adic completion of $\mathbb{P}^1$ over $\mathrm{Spf}\mathcal{O}_k$ with a coordinate $x$, which is a smooth lifting of $\mathbb{P}_\kappa^1$.
	
	Over the frame $(X,\mathbb{P}_\kappa^1,\hat{\mathbb{P}}^1)$ (see \cite{Be96}), the rank $n$ \textit{Bessel overconvergent} $F$\textit{-isocrystal} $\mathrm{Be}_n^\dagger$ on $X/k$, has the underlying connection given by
	\[d-\begin{pmatrix}
		0 & 1 & 0 & \cdots & 0 \\
		0 & 0 & 1 & \cdots & 0 \\
		\vdots & \vdots & \ddots & \ddots & \vdots \\
		0 & 0 & 0 & \cdots & 1 \\
		\pi^nx & 0 & 0 & \cdots & 0 
	\end{pmatrix}\frac{dx}{x}.\]
	Its Frobenius structure is constructed in \cite{Dw74} (3.9), \cite{Sp77} Theorem 1.3.9 and \cite{X-Z22} Theorem 4.4.4.
	
	At the point $\infty$, $\mathrm{Be}_n^\dagger$ induces a rank $n$ differential module $M$ over $\mathcal{R}$ with a Frobenius structure, whose associated differential equation is the following rank $n$ \textit{Bessel equation}
	\begin{equation}
		(x\frac{d}{dx})^n-\frac{(-\pi)^n}{x}=0.
	\end{equation}
	
	According to the $p$-adic local monodromy theorem (see \cite{An02F,Ked04,Meb02}), a differential module over $\mathcal{R}$ with a Frobenius structure is quasi-unipotent, hence we can talk about the $p$-adic slopes of $M$.
	
	\begin{prop}
		Let $\alpha_1,\cdots,\alpha_n$ (resp. $\beta_1,\cdots,\beta_n$) be the $p$-adic slopes (resp. formal slopes) of $M$, then we have
		\[\alpha_1=\cdots=\alpha_n=\beta_1=\cdots=\beta_n=\frac{1}{n}.\]
	\end{prop}
	
	\begin{proof}
		When $\rho<\omega^n$, we compute the generic radius of $M$ as follows (\cite{C-M93}, Proposition 2.1.1).
		\[R_1(M,\rho)=\omega\rho\Big|\frac{\pi^n}{x}\Big|_\rho^{-\frac{1}{n}}=\rho^{1+\frac{1}{n}}.\]
		
		The solvability of $M$ and the convexity of $F_1(M,r)$ then imply that $R_1(M,\rho)=\rho^{1+\frac{1}{n}}$ for all $\rho\in(0,1)$. We therefore conclude that $\alpha_1=\frac{1}{n}$.
		
		Since the irregularity of $M$ is an integer by Proposition 4.12, the only possibility is that
		\[\alpha_1=\cdots=\alpha_n=\frac{1}{n},\quad\mathrm{Irr}(M)=\sum_{i=1}^{n}\alpha_i=1.\]
		
		Rewrite (6.1.1) in the form
		\begin{equation}
			(\frac{d}{dx})^n+\frac{a_{n-1}}{x}(\frac{d}{dx})^{n-1}+\cdots+\frac{a_1}{x^{n-1}}\frac{d}{dx}-\frac{(-\pi)^n}{x^{n+1}}=0,\quad a_1,\cdots,a_{n-1}\in\mathbb{Z}.
		\end{equation}
	
		The formal Newton polygon associated with (6.2.1) is a segment with endpoints at $(-n,0)$ and $(0,-n-1)$. We then conclude by Lemma 3.4 that
		\[\beta_1=\cdots=\beta_n=\frac{1}{n}.\qedhere\]
	\end{proof}
	
	\begin{rem}
		The inequalities in (1.1.1) are all equalities in this case.
	\end{rem}
	
	Depending on whether $n$ and $p$ are coprime, the image of $\rho_M$ behaves in different ways. In Section 6.3, we determine the image of $\rho_M$ when $(n,p)=1$. In Section 6.4, we focus on a special case where $n=p=2$. In particular, the adjoint module $\mathrm{Ad}(M)$ in this case is an example where the inequality (1.1.1) is strict unless $i=n$.
	
	\subsection{Image of $\rho_M$: coprime case}
	
	\begin{nota}
		When $(n,p)=1$, let $a$ be the order of $p$ in $(\mathbb{Z}/n\mathbb{Z})^\times$ and write $q=p^a$. Fix a primitive $n$-th root of unity $\zeta\in k^\mathrm{ur}$ and let $\overline{\zeta}\in\overline{\kappa}$ be its reduction. When $n$ is even, fix a primitive $2n$-th root of unity $\eta\in k^\mathrm{ur}$ such that the reduction $\overline{\eta}^2=\overline{\zeta}$ in $\overline{\kappa}$.
	\end{nota}
	
	We first describe the image of $\rho_M$.
	
	\begin{con}
		When $n$ is odd, we construct a tower of extensions
		\[\overline{\kappa}((x))\subseteq E=\overline{\kappa}((y))\subseteq E(w),\quad x=y^n,w^q-w=\frac{1}{y}.\]
		
		Note that $E(w)/\overline{\kappa}((x))$ is a Galois extension because $E(w)$ is the splitting field of the polynomial $f(T)=(T^q-T)^n-\frac{1}{x}$ over $\overline{\kappa}((x))$. In fact, all of the roots of $f(T)$ are $\overline{\zeta}^iw+\alpha$, where $i=0,1,\cdots,n-1$ and $\alpha\in\mathbb{F}_q$.
		
		Let $G$ be the Galois group of the extension $E(w)/\overline{\kappa}((x))$, then $G$ is isomorphic to the semidirect product of $\mathbb{F}_q$ and $\mathbb{Z}/n\mathbb{Z}$, where the action of $i\in\mathbb{Z}/n\mathbb{Z}$ on $\mathbb{F}_q$ is multiplication by $\overline{\zeta}^{-i}$. To be precise, the isomorphism is given by sending $(\alpha,i)\in\mathbb{F}_q\rtimes\mathbb{Z}/n\mathbb{Z}$ to $\sigma\in G$ such that
		\[\sigma(y)=\overline{\zeta}^iy,\quad\sigma(w)=\overline{\zeta}^{-i}(w+\alpha).\]
		
		The wild inertia subgroup of $G$ is
		\[G_1=\mathrm{Gal}\Big(E(w)/\overline{\kappa}((y))\Big)\simeq\mathbb{F}_q,\]
		and the lower number ramification filtration of $G$ is given by
		\[G=G_0\supseteq G_1\supseteq G_2=\{1\}.\]
		The upper number ramification filtration of $G$ admits a unique jump at $\frac{1}{n}$.
	\end{con}
	
	\begin{con}
		When $n$ is even, we construct a tower of extensions
		\[\overline{\kappa}((x))\subseteq E=\overline{\kappa}((y))\subseteq E(w),\quad x=y^{2n},w^q-w=\frac{1}{y^2}.\]
		
		Note that $E(w)/\overline{\kappa}((x))$ is a Galois extension because $E(w)$ is the splitting field of the polynomial $g(T)=f(T)(T^{2n}-x)$ over $\overline{\kappa}((x))$, where $f(T)$ is the polynomial in Construction 6.5.
		
		Let $G$ be the Galois group of the extension $E(w)/\overline{\kappa}((x))$, then $G$ is isomorphic to the semidirect product of $\mathbb{F}_q$ and $\mathbb{Z}/2n\mathbb{Z}$, where the action of $i\in\mathbb{Z}/2n\mathbb{Z}$ on $\mathbb{F}_q$ is multiplication by $\overline{\zeta}^{-i}$. To be precise, the isomorphism is given by sending $(\alpha,i)\in\mathbb{F}_q\rtimes\mathbb{Z}/2n\mathbb{Z}$ to $\sigma\in G$ such that
		\[\sigma(y)=\overline{\eta}^iy,\quad\sigma(w)=\overline{\zeta}^{-i}(w+\alpha).\]
		
		The wild inertia subgroup $G_1$ is isomorphic to $\mathbb{F}_q$, and the lower number ramification filtration of $G$ is given by
		\[G=G_0\supseteq G_1=G_2\supseteq G_3=\{1\}.\]
		The upper number ramification filtration of $G$ admits a unique jump at $\frac{1}{n}$.
	\end{con}
	
	\begin{prop}
		When $n$ is odd (resp. even), the image of $\rho_M$ is isomorphic to the Galois group $G$ in Construction 6.5 (resp. Construction 6.6).
	\end{prop}

	\begin{rem}
		Let $\ell$ be a prime number different to $p$. 
		The Bessel $F$-isocrystal is the $p$-adic companion of the $\ell$-adic Kloosterman sheaf $\Kl_n$ introduced by Deligne in \cite{Del77} in the sense of \cite{Abe18,Del80} (see \cite{X-Z22} Theorem 1.1.4). 
		Via an isomorphism $\overline{\mathbb{Q}}_p\simeq \overline{\mathbb{Q}}_{\ell}$, 
		the local monodromy representation $\rho_M$ associated to $\Be_{n,\infty}^{\dagger}$ is isomorphic to that of $\Kl_{n,\infty}$ (see \cite{Del73} Th\'eor\`eme 9.8, \cite{K-Z25} Theorem 4.4.5). 
		
		In the $\ell$-adic case, under the assumption that $(n,p)=1$, Fu and Wan explicitly calculated the local monodromy representation of $\Kl_{n,\infty}$ (\cite{F-W05}, Theorem 1.1), which is compatible with ours in Proposition 6.7. 

		In a forthcoming work of Xu and Yi \cite{X-Y}, the authors study the local monodromy representation associated to the Bessel $F$-isocrystal for reductive groups at $\infty$ via the theory of $p$-adic differential equations.   
	\end{rem}
	
	\begin{rem}
		Let $m=\frac{q-1}{n}$ and $\omega$ be a generator of $\mathbb{F}_q^\times$ such that $\omega^m=\overline{\zeta}$. Fix a nontrivial linear character $\psi:\mathbb{F}_q\rightarrow(k^\mathrm{ur})^\times$.
		
		When $n$ is odd, the character table of $G\simeq\mathbb{F}_q\rtimes\mathbb{Z}/n\mathbb{Z}$ is listed as follows. The character of $\rho_M$ is one of the $\chi_\ell$.
		\begin{table}[htbp]
			\centering
			\begin{tabular}{|c|c|c|c|}
				\hline
				conjugate class & $(0,0)$ & $(\omega^i,0)\quad i=1,\cdots,m$ & $(0,j)\quad j=1,\cdots,n-1$ \\ 
				\hline
				$\phi_t\quad t=1,\cdots,n$ & $1$ & $1$ & $\zeta^{jt}$ \\
				\hline
				$\chi_\ell\quad \ell=1,\cdots,m$ & $n$ & $\sum_{s=1}^{n}\psi(\omega^{i+\ell}\overline{\zeta}^s)$ & $0$ \\
				\hline
			\end{tabular}
		\end{table}
		
		When $n$ is even, the character table of $G\simeq\mathbb{F}_q\rtimes\mathbb{Z}/2n\mathbb{Z}$ is listed as follows. The character of $\rho_M$ is one of the $\chi_\ell$ or $\lambda_\ell$.
		\begin{table}[htbp]
			\centering
			\begin{tabular}{|c|c|c|c|c|c|}
				\hline
				conjugate class & $(0,0)$ & $(\omega^i,0)\quad i=1,\cdots,m$ & $(0,n)$ & $(\omega^i,n)\quad i=1,\cdots,m$ & $(0,j)\quad j\neq0,n$ \\ 
				\hline
				$\phi_t\quad t=1,\cdots,2n$ & $1$ & $1$ & $\eta^{nt}$ & $\eta^{nt}$ & $\eta^{jt}$ \\
				\hline
				$\chi_\ell\quad \ell=1,\cdots,m$ & $n$ & $\sum_{s=1}^{n}\psi(\omega^{i+\ell}\overline{\zeta}^s)$ & $n$ & $\sum_{s=1}^{n}\psi(\omega^{i+\ell}\overline{\zeta}^s)$ & $0$ \\
				\hline
				$\lambda_\ell\quad \ell=1,\cdots,m$ & $n$ & $\sum_{s=1}^{n}\psi(\omega^{i+\ell}\overline{\zeta}^s)$ & $-n$ & $-\sum_{s=1}^{n}\psi(\omega^{i+\ell}\overline{\zeta}^s)$ & $0$ \\
				\hline
			\end{tabular}
		\end{table}
	\end{rem}

	We will prove Proposition 6.7 by solving $p$-adic differential equations.
	
	\begin{lemma}
		$\mathrm{(}$\cite{Sp77} $\mathrm{Proposition}~5.1.7)$ Substitute $x$ by $y^n$ in (6.1.1), the equation is now
		\begin{equation}
			(y\frac{d}{dy})^n-\frac{(-n\pi)^n}{y^n}=0.
		\end{equation}
		Solutions to (6.10.1) are of the form
		\begin{equation}
			w_i(y)=y^\frac{n-1}{2}v_i(y)\mathrm{exp}(-\frac{\zeta^in\pi}{y}),\quad i=0,1,\cdots,n-1,
		\end{equation}
		where $v_i(y)$ converges in the open unit disc.
	\end{lemma} 
	
	\begin{lemma}
		Let $A_i$ denote the Artin-Schreier extension $\overline{\kappa}((\frac{1}{z}))/\overline{\kappa}((y))$ defined by
		\[z^p-z=\frac{\overline{\zeta}^i}{y},\quad i=0,1,\cdots,n-1.\]
		Then the composition of $A_i$ for $i=0,1,\cdots,n-1$ is $A=\overline{\kappa}((\frac{1}{w}))/\overline{\kappa}((y))$ defined by
		\[w^q-w=\frac{1}{y}.\]
		Recall that $q=p^a$ and $a$ is the order of $p$ in $(\mathbb{Z}/n\mathbb{Z})^\times$.
	\end{lemma}
	
	\begin{proof}
		Let $w_0$ be a root of $w^q-w=\frac{1}{y}$ and write $z_i=\sum_{k=0}^{a-1}w_0^{p^k}\overline{\zeta}^{ip^k}$ for $i=0,1,\cdots,n-1$. Compute
		\[z_i^p-z_i=\sum_{k=0}^{a-1}w_0^{p^{k+1}}\overline{\zeta}^{ip^{k+1}}-\sum_{k=0}^{a-1}w_0^{p^k}\overline{\zeta}^{ip^k}=w_0^q\overline{\zeta}^{qi}-w_0\overline{\zeta}^i=\overline{\zeta}^i(w_0^q-w_0)=\frac{\overline{\zeta}^i}{y}.\]
		Hence each $A_i$ is contained in $A=\overline{\kappa}((\frac{1}{w}))$.

		Now consider the linear equation
		\[(z_0\quad z_1\quad\cdot\cdot\cdot\quad z_{a-1})=(w_0\quad w_0^p\quad\cdot\cdot\cdot\quad w_0^{p^{a-1}})B,\]
		where $B=(b_{kj})_{a\times a}$ is an $a\times a$ matrix with $b_{kj}=\overline{\zeta}^{(j-1)p^{k-1}}$. The determinant of $B$ is the Vandermonde determinant
		\[\mathrm{det}(B)=\prod_{1\leq k<j\leq a}(\overline{\zeta}^{p^{j-1}}-\overline{\zeta}^{p^{k-1}}).\]
		When $j\neq k$, we find that $\overline{\zeta}^{p^{j-1}}\neq\overline{\zeta}^{p^{k-1}}$ from the definition of $a$. Hence $\mathrm{det}(B)\neq 0$ and $w_0$ is a $k$-linear combination of $z_i$. We thus conclude that $A=\overline{\kappa}((\frac{1}{w}))$ is the composition of $A_i$ for $i=0,1,\cdots,n-1$.
	\end{proof}
	
	We are now ready to prove Proposition 6.7.
	
	\begin{proof}
		We first give the proof in the case where $n$ is odd.
		
	    Consider the tamely ramified extension $\overline{\kappa}((y))/\overline{\kappa}((x))$ such that $x=y^n$. The associated extension of $\mathcal{R}$ is denoted by $\mathcal{R}_y$, where $y$ is the variable satisfying $x=y^n$.
		
		Let $A_i$ be the Artin-Schreier extension of $\overline{\kappa}((y))$ in Lemma 6.11. The associated extension of $\mathcal{R}_y$ is denoted by $\mathcal{R}_\frac{1}{z_i}$, where $\frac{1}{z_i}$ is the variable satisfying $z_i^p-z_i=\frac{\zeta^i}{y}$.
		
		Now we have the following equality in $\mathcal{R}_\frac{1}{z_i}$.
		\[\mathrm{exp}(-\frac{\zeta^i\pi}{y})=\mathrm{exp}(\pi(z_i-z_i^p)).\]
		The Dwork exponential series $\mathrm{exp}(\pi(z_i-z_i^p))$ converges when $|z_i|<p^\frac{p-1}{p^2}$ (see \cite{Rob00}, Chapter 7.2.4). Hence $\mathrm{exp}(-\frac{\zeta^i\pi}{y})$ belongs to $\mathcal{R}_\frac{1}{z_i}$, and so does $\mathrm{exp}(-\frac{\zeta^in\pi}{y})$.
		
		Since $n$ is odd, the solution $w_i(y)$ in (6.10.2) belongs to the extended Robba ring $\mathcal{R}_\frac{1}{z_i}$.
		
		Now consider the extension $A/\overline{\kappa}((y))$ in Lemma 6.11. The associated extension of $\mathcal{R}_y$ is denoted by $\mathcal{R}_\frac{1}{w}$, where $\frac{1}{w}$ is the variable satisfying $w^q-w=\frac{1}{y}$. According to Lemma 6.11, all of the solutions $w_i(y)$ in (6.10.2) belong to $\mathcal{R}_\frac{1}{w}$, i.e. we can solve (6.1.1) in the extension of $\mathcal{R}$ associated with $A/\overline{\kappa}((x))$.
		
		Note that $A/\overline{\kappa}((x))$ is exactly the extension $E(w)/\overline{\kappa}((x))$ in Construction 6.5. Hence the image of $\rho_M$ is a quotient of $G=\mathrm{Gal}\Big(E(w)/\overline{\kappa}((x))\Big)$.
		
		According to Lemma 6.12, any quotient of $G$ is isomorphic to either $G$ itself or a quotient of $\mathbb{Z}/n\mathbb{Z}$. The upper number ramification filtration of any quotient of $\mathbb{Z}/n\mathbb{Z}$ is trivial and has no jumps. However, by Corollary 6.1 and Proposition 6.2, the upper number ramification filtration of the image of $\rho_M$ should admit a jump at $\frac{1}{n}$. We then conclude that the image of $\rho_M$ coincides with $G$.
		
		When $n$ is even, we conclude by a similar argument.
	\end{proof}
	
	\begin{lemma}
		When $n$ is odd (resp. even), let $G$ be the Galois group in Construction 6.5 (resp. Construction 6.6). Let $H$ be a quotient of $G$, then $H$ is isomorphic to either $G$ itself or a quotient of $\mathbb{Z}/n\mathbb{Z}$ (resp. $\mathbb{Z}/2n\mathbb{Z}$).
	\end{lemma}
	
	\begin{proof}
		
		Let $G'$ be the kernel of the canonical map $G\rightarrow H$, then $G'$ is a normal subgroup of $G$. Let $G_1\simeq\mathbb{F}_q$ be the wild inertia subgroup of $G$.
		
		When $n$ is odd, we view elements of $G$ as pairs $(\alpha,i)\in\mathbb{F}_q\rtimes\mathbb{Z}/n\mathbb{Z}$ via the isomorphism $G\simeq\mathbb{F}_q\rtimes\mathbb{Z}/n\mathbb{Z}$ in Construction 6.5.
		
		\textit{Case 1.} $G'\subseteq G_1$.
		
		In this case, we show that $G'$ is either trivial or equal to $G_1$, i.e. $G_1$ is an irreducible $\mathbb{F}_p[H]$-module.
		
		If $G'$ is nontrivial, then there exists $(\alpha,0)\in G'$ such that $\alpha\in\mathbb{F}_q^\times$. For any $i\in\mathbb{Z}/n\mathbb{Z}$, we have
		\[(0,i)(\alpha,0)(0,i)^{-1}=(\overline{\zeta}^i\alpha,0)\in G'.\]
		
		Therefore, we obtain an inclusion
		\[\Big\{(\theta\alpha,0)~\Big|~\theta\in\mathbb{F}_p[\overline{\zeta}]\Big\}\subseteq G',\]
		where $\mathbb{F}_p[\overline{\zeta}]\simeq\mathbb{F}_p[T]/(T^n-1)$. 
		
		Recall that in Notation 6.4, $q=p^a$ where $a$ is the order of $p$ in $(\mathbb{Z}/n\mathbb{Z})^\times$. Note that $a$ is also the degree of the minimal polynomial of $\overline{\zeta}$ over $\mathbb{F}_p$, hence we have \[\mathrm{dim}_{\mathbb{F}_p}\mathbb{F}_p[\overline{\zeta}]=\mathrm{dim}_{\mathbb{F}_p}\mathbb{F}_q=a.\]
		We then conclude that $G'=G_1\simeq\mathbb{F}_q$ and $H\simeq\mathbb{Z}/n\mathbb{Z}$. 
		
		\textit{Case 2.} $G'$ is not contained in $G_1$.
		
		In this case, we show that $G_1\subseteq G'$.
		
		Pick an element $(\alpha,i)\in G'-G_1$ where $i\neq0$. For any $\beta\in\mathbb{F}_q$, we have
		
		\[(\beta,0)(\alpha,i)(\beta,0)^{-1}=(\alpha+(1-\overline{\zeta}^{-i})\beta,i)\in G'.\]
		
		Let $\beta$ run through $\mathbb{F}_q$, then for any $\gamma\in\mathbb{F}_q$ we have $(\gamma,i)\in G'$. Fixing some $\gamma\neq\alpha$, we compute
		\begin{equation}
			(\alpha,i)^{-1}(\gamma,i)=(\overline{\zeta}^i(\gamma-\alpha),0)\in G'\cap G_1.
		\end{equation}
		
		Clearly $G'\cap G_1$ is a normal subgroup of $G$ contained in $G_1$, then $G'\cap G_1$ is either trivial or equal to $G_1$ by the argument in Case 1. Now (6.12.1) shows that $G'\cap G_1$ is nontrivial, hence $G'\cap G_1=G_1$. We conclude that $H\simeq G/G'$ is a quotient of $G/G_1\simeq\mathbb{Z}/n\mathbb{Z}$.
		
		When $n$ is even, we conclude by a similar argument.
	\end{proof}
	
	\subsection{Image of $\rho_M$: case $p|n$}
	
	When $p|n$, the method in Section 6.3 fails because the extension $\overline{\kappa}((y))/\overline{\kappa}((x))$ is inseparable where $x=y^n$. Moreover, the function $v_i(y)$ in (6.10.2) may not have a desired radius of convergence so that it no longer belongs to the Robba ring.
	
	In general, the image of $\rho_M$ is not known in the literature. Andr\'e determined the image of $\rho_M$ when $n=p=2$ (\cite{An02R}, Theorem 8.2), and Qin determined the image of $\rho_M$ when $n=p=3$ (\cite{Qin24}, Theorem 2.24).
	
	In this subsection, we introduce the result of Andr\'e. Let $n=p=2$, the differential equation (6.1.1) now becomes
	\[(x\frac{d}{dx})^2-\frac{4}{x}=0.\]
	
	\begin{prop}
		$\mathrm{(}$\cite{An02R}$\mathrm{,~Theorem~8.2)}$ When $n=p=2$, let $G$ be the image of $\rho_M$.
		
		(1) $G$ is isomorphic to $\mathrm{SL}_2(\mathbb{F}_3)$, and it fits into an exact sequence
		\[1\rightarrow G_1\rightarrow G\rightarrow\mathbb{Z}/3\mathbb{Z}\rightarrow1,\]
		where the wild inertia subgroup $G_1$ is isomorphic to the quaternionic group $\{\pm1,\pm i,\pm j,\pm k\}$.
		
		(2) The lower (resp. upper) number ramification filtration of $G$ is
		\[G=G_0\supseteq G_1\supseteq G_2=G_3\supseteq G_4=\{1\},\quad(\mathrm{resp.}~G=G^0\supseteq G^\frac{1}{3}\supseteq G^\frac{1}{2}\supseteq\{1\}),\]
		where $G_1=G^\frac{1}{3}$ (resp. $G_2=G^\frac{1}{2}$) is isomorphic to the quaternionic group (resp. $\{\pm1\}$).
	\end{prop}
	
	\begin{rem}
		The ramification filtration of $G$ in Proposition 6.13 is determined as follows. We first note that nontrivial ramification subgroups of $G$ are isomorphic to either the quaternionic group or $\{\pm1\}$. Meanwhile, the upper number ramification filtration of $G$ admits a jump at $\frac{1}{2}$ by Corollary 6.1 and Proposition 6.2. These two conditions then uniquely determine the ramification filtration of $G$.
	\end{rem}
	
	Denote the trace-zero component of $\mathrm{End}_\mathcal{R}(M)$ by $\mathrm{Ad}(M)$. The differential equation associated with $\mathrm{Ad}(M)$ is
	\[(x\frac{d}{dx})^3-\frac{16}{x}(x\frac{d}{dx})+\frac{8}{x}=0.\]
	
    Proposition 6.13 shows that the upper number ramification filtration of $G$ admits two jumps at $\frac{1}{3}$ and $\frac{1}{2}$ respectively. However, only the value $\frac{1}{2}$ appears as a $p$-adic slope of $M$. We will see that $\frac{1}{3}$ appears as a $p$-adic slope of $\mathrm{Ad}(M)$ by the following proposition. This is an example where the inequality (1.1.1) is strict except for $i=n$.
	
	\begin{prop}
		Let $\alpha_1,\alpha_2,\alpha_3$ (resp. $\beta_1,\beta_2,\beta_3$) be the $p$-adic slopes (resp. formal slopes) of $\mathrm{Ad}(M)$. Then we have
		\[\alpha_1=\alpha_2=\alpha_3=\frac{1}{3},\quad\beta_1=\beta_2=\frac{1}{2}, \beta_3=0.\]
	\end{prop}
	
	\begin{proof}
		We apply the same method as in Proposition 6.2.
		
		We first compute that
		\[
		R_1(M,\rho)=
		\begin{cases}
			\rho^\frac{4}{3}, & \frac{1}{64}<\rho<1; \\
			2\rho^\frac{3}{2}, & 0<\rho<\frac{1}{64}.
		\end{cases}
	    \quad
		F_1(M,r)=
		\begin{cases}
		  	\frac{4}{3}r, & 0<r<6\mathrm{log}2; \\
		  	\frac{3}{2}r-\mathrm{log}2, & r>6\mathrm{log}2.
		\end{cases}
		\]
		We then conclude that $\alpha_1=\alpha_2=\alpha_3=\frac{1}{3}$.
		
		Rewrite the differential equation associated to $\mathrm{Ad}(M)$ in the form
		\begin{equation}
			(\frac{d}{dx})^3+\frac{3}{x}(\frac{d}{dx})^2+\frac{x-16}{x^3}\frac{d}{dx}+\frac{8}{x^4}=0.
		\end{equation}
		
		The formal Newton polygon associated with (6.15.1) is drawn in Figure 5 below. The multi-set of slopes of the formal Newton polygon is $\{-\frac{3}{2},-\frac{3}{2},-1\}$. By Lemma 3.4, we conclude that $\beta_1=\beta_2=\frac{1}{2},\beta_3=0$.
		\end{proof}
		
		\begin{center}
			\begin{figure}[htbp]
				\begin{tikzpicture}
					\draw[->] (-4,0) -- (0.5,0) node[right] {$x$};
					\draw[->] (0,-3.5) -- (0,0.5) node[above] {$y$};
					\node[below left] at (0,0) {$0$};
					\node[above] at (-3,0) {$-3$};
					\node[right] at (0,-3.2) {$-4$};
					\node[above] at (-1,0) {$-1$};
					\node[right] at (0,-2.4) {$-3$};
					\filldraw[black] (-3,0) circle (1pt); 
					\filldraw[black] (-1,-2.4) circle (1pt); 
					\filldraw[black] (0,-3.2) circle (1pt); 
					\filldraw[black] (-2,-0.8) circle (1pt);
					\draw[black] (-3,0) -- (-1,-2.4) -- (0,-3.2);
					\draw[loosely dashed] (-1,0) -- (-1,-2.4);
					\draw[loosely dashed] (0,-2.4) -- (-1,-2.4);
				\end{tikzpicture}
				\caption{}
			\end{figure}
		\end{center}


\begin{thebibliography}{100}
		\bibitem{Abe18}T. Abe, Langlands correspondence for isocrystals and the existence of crystalline companions for curves, J. Amer. Math. Soc. {\bf 31} (2018), no.~4, 921--1057.
		
		\bibitem{An02F}Y. Andr\'e, Filtrations de type Hasse-Arf et monodromie $p$-adique, Invent. Math. {\bf 148} (2002), no.~2, 285--317.
		
		\bibitem{An02R}Y. Andr\'e, Repr\'esentations galoisiennes et op\'erateurs de Bessel $p$-adiques, Ann. Inst. Fourier (Grenoble) {\bf 52} (2002), no.~3, 779--808.
		
		\bibitem{Ba82}F. Baldassarri, Differential modules and singular points of $p$-adic differential equations, Adv. in Math. {\bf 44} (1982), no.~2, 155--179.
		
		\bibitem{Be96}P. Berthelot, {\it Cohomologie rigide et cohomologie rigide \`a supports propres}, IRMAR, 1996.
		
		\bibitem{C-D94}G. Christol and B.~M. Dwork, Modules diff\'erentiels sur des couronnes, Ann. Inst. Fourier (Grenoble) {\bf 44} (1994), no.~3, 663--701.
		
		\bibitem{C-M93}G. Christol and Z. Mebkhout, Sur le th\'eor\`eme de l'indice des \'equations diff\'erentielles $p$-adiques. I, Ann. Inst. Fourier (Grenoble) {\bf 43} (1993), no.~5, 1545--1574.
		
		\bibitem{C-M97}G. Christol and Z. Mebkhout, Sur le th\'eor\`eme de l'indice des \'equations diff\'erentielles $p$-adiques. II, Ann. of Math. (2) {\bf 146} (1997), no.~2, 345--410.
		
		\bibitem{C-M00}G. Christol and Z. Mebkhout, Sur le th\'eor\`eme de l'indice des \'equations diff\'erentielles $p$-adiques. III, Ann. of Math. (2) {\bf 151} (2000), no.~2, 385--457.
		
		\bibitem{C-M01}G. Christol and Z. Mebkhout, Sur le th\'eor\`eme de l'indice des \'equations diff\'erentielles $p$-adiques. IV, Invent. Math. {\bf 143} (2001), no.~3, 629--672.
	
	    \bibitem{Del73}P. Deligne, Les constantes des \'equations fonctionnelles des fonctions $L$, in {\it Modular functions of one variable, II (Proc. Internat. Summer School, Univ. Antwerp, Antwerp, 1972)}, pp. 501--597, Lecture Notes in Math., Vol. 349, Springer, Berlin-New York.

	    \bibitem{Del77}P. Deligne, Applications de la formule des traces aux sommes trigonom\'etriques, in {\it Cohomologie \'etale}, 168--232, Lecture Notes in Math., 569, Springer, Berlin.
	    
	    \bibitem{Del80}P. Deligne, La conjecture de Weil. II, Inst. Hautes \'Etudes Sci. Publ. Math. No. 52 (1980), 137--252.

		\bibitem{Dw74}B.~M. Dwork, Bessel functions as $p$-adic functions of the argument, Duke Math. J. {\bf 41} (1974), 711--738.
		
		\bibitem{D-G-S94}B.~M. Dwork, G. Gerotto and F.~J. Sullivan, {\it An introduction to $G$-functions}, Annals of Mathematics Studies, 133, Princeton Univ. Press, Princeton, NJ, 1994.
		
		\bibitem{F-W05}L. Fu and D.~Q. Wan, $L$-functions for symmetric products of Kloosterman sums, J. Reine Angew. Math. {\bf 589} (2005), 79--103.
		
		\bibitem{Ka87}N.~M. Katz, On the calculation of some differential Galois groups, Invent. Math. {\bf 87} (1987), no.~1, 13--61.
		
		\bibitem{Ka88}N.~M. Katz, {\it Gauss sums, Kloosterman sums, and monodromy groups}, Annals of Mathematics Studies, 116, Princeton Univ. Press, Princeton, NJ, 1988.
		
		\bibitem{Ked04}K.~S. Kedlaya, A $p$-adic local monodromy theorem, Ann. of Math. (2) {\bf 160} (2004), no.~1, 93--184.
		
		\bibitem{Ked05}K.~S. Kedlaya, Local monodromy of $p$-adic differential equations: an overview, Int. J. Number Theory {\bf 1} (2005), no.~1, 109--154.
		
		\bibitem{Ked22}K.~S. Kedlaya, {\it $p$-adic differential equations}, second edition [of MR 2663480], 
		Cambridge Studies in Advanced Mathematics, [199], Cambridge Univ. Press, Cambridge, 2022.
	
	    \bibitem{K-Z25} M. Kisin and R. Zhou, \textit{Strongly compatible systems associated to semistable abelian varieties}, preprint (2025) https://arxiv.org/abs/2505.02165v1
		
		\bibitem{Le75}A.~H.~M. Levelt, Jordan decomposition for a class of singular differential operators, Ark. Mat. {\bf 13} (1975), 1--27.
		
		\bibitem{Meb02}Z. Mebkhout, Analogue $p$-adique du th\'eor\`eme de Turrittin et le th\'eor\`eme de la monodromie $p$-adique, Invent. Math. {\bf 148} (2002), no.~2, 319--351.
		
		\bibitem{Pu04}A. Pulita, Frobenius structure for rank one $p$-adic differential equations, in {\it Ultrametric functional analysis}, 247--258, Contemp. Math., 384, Amer. Math. Soc., Providence, RI.
		
		\bibitem{Pu15}A. Pulita, The convergence Newton polygon of a $p$-adic differential equation I: Affinoid domains of the Berkovich affine line, Acta Math. {\bf 214} (2015), no.~2, 307--355.
		
		\bibitem{PP15}J. Poineau and A. Pulita, The convergence Newton polygon of a $p$-adic differential equation II: Continuity and finiteness on Berkovich curves, Acta Math. {\bf 214} (2015), no.~2, 357--393.
		
		\bibitem{PP24}J. Poineau and A. Pulita, \textit{The convergence Newton polygon of a $p$-adic differential equation V : local index theorems}, preprint (2024) https://arxiv.org/abs/1309.3940v3
		
		\bibitem{Qin24}Y. Qin, $L$-functions of Kloosterman sheaves, Proc. Lond. Math. Soc. (3) {\bf 129} (2024), no.~5, Paper No. e70003, 60 pp.
		
		\bibitem{Rob00}A.~M. Robert, {\it A course in $p$-adic analysis}, Graduate Texts in Mathematics, 198, Springer, New York, 2000.
		
		\bibitem{Se79}J.-P. Serre, {\it Local fields}, translated from the French by Marvin Jay Greenberg, 
		Graduate Texts in Mathematics, 67, Springer, New York-Berlin, 1979.
		
		\bibitem{Sp77}S.~I. Sperber, $p$-adic hypergeometric functions and their cohomology, Duke Math. J. {\bf 44} (1977), no.~3, 535--589.
		
		\bibitem{Tsu98}N. Tsuzuki, The local index and the Swan conductor, Compositio Math. {\bf 111} (1998), no.~3, 245--288.
		
		\bibitem{Tu55}H.~L. Turrittin, Convergent solutions of ordinary linear homogeneous differential equations in the neighborhood of an irregular singular point, Acta Math. {\bf 93} (1955), 27--66.
		
		\bibitem{X-Y}D. Xu and L. Yi, Frobenius structure on rigid connections and arithmetic applications, to appear.
		
		\bibitem{X-Z22}D. Xu and X. Zhu, Bessel $F$-isocrystals for reductive groups, Invent. Math. {\bf 227} (2022), no.~3, 997--1092.
	\end{thebibliography}
\end{document}